\documentclass[12pt,reqno,draft]{amsart}
\usepackage{amssymb,amscd,url}

\begin{document}

\allowdisplaybreaks
\hyphenation{non-archi-me-dean}


\newtheorem{theorem}{Theorem}
\newtheorem{lemma}[theorem]{Lemma}
\newtheorem{sublemma}[theorem]{Sublemma}
\newtheorem{conjecture}[theorem]{Conjecture}
\newtheorem{proposition}[theorem]{Proposition}
\newtheorem{propositionnonum}[theorem]{Proposition}
\newtheorem{corollary}[theorem]{Corollary}
\newtheorem{claim}[theorem]{Claim}

\theoremstyle{definition}
\newtheorem*{definition}{Definition}
\newtheorem{remark}[theorem]{Remark}
\newtheorem{example}[theorem]{Example}
\newtheorem{question}[theorem]{Question}

\theoremstyle{remark}
\newtheorem*{acknowledgement}{Acknowledgements}


\newenvironment{notation}[0]{%
  \begin{list}%
    {}%
    {\setlength{\itemindent}{0pt}
     \setlength{\labelwidth}{4\parindent}
     \setlength{\labelsep}{\parindent}
     \setlength{\leftmargin}{5\parindent}
     \setlength{\itemsep}{0pt}
     }%
   }%
  {\end{list}}

\newenvironment{parts}[0]{%
  \begin{list}{}%
    {\setlength{\itemindent}{0pt}
     \setlength{\labelwidth}{1.5\parindent}
     \setlength{\labelsep}{.5\parindent}
     \setlength{\leftmargin}{2\parindent}
     \setlength{\itemsep}{0pt}
     }%
   }%
  {\end{list}}
\newcommand{\Part}[1]{\item[\upshape#1]}

\renewcommand{\a}{\alpha}
\newcommand{\aupper}{\overline{\alpha}}
\newcommand{\alower}{\underline{\alpha}}
\renewcommand{\b}{\beta}
\newcommand{\bupper}{\overline{\beta}}
\newcommand{\blower}{\underline{\beta}}
\newcommand{\bfbeta}{{\boldsymbol{\beta}}}
\newcommand{\g}{\gamma}
\renewcommand{\d}{\delta}
\newcommand{\e}{\epsilon}
\newcommand{\bfepsilon}{\boldsymbol{\epsilon}}
\newcommand{\f}{\varphi}
\newcommand{\bfphi}{{\boldsymbol{\f}}}
\renewcommand{\k}{\kappa}
\renewcommand{\l}{\lambda}
\newcommand{\bfl}{{\boldsymbol{\lambda}}}
\newcommand{\lhat}{\hat\lambda}
\newcommand{\lplus}{\lambda^{\scriptscriptstyle+}}
\newcommand{\m}{\mu}
\newcommand{\bfmu}{{\boldsymbol{\mu}}}
\renewcommand{\o}{\omega}
\renewcommand{\r}{\rho}
\newcommand{\rbar}{{\bar\rho}}
\newcommand{\s}{\sigma}
\newcommand{\sbar}{{\bar\sigma}}

\renewcommand{\t}{\tau}
\newcommand{\z}{\zeta}

\newcommand{\D}{\Delta}
\newcommand{\G}{\Gamma}
\newcommand{\F}{\Phi}
\renewcommand{\L}{\Lambda}
\newcommand{\ga}{{\mathfrak{a}}}
\newcommand{\gb}{{\mathfrak{b}}}
\newcommand{\gn}{{\mathfrak{n}}}
\newcommand{\gp}{{\mathfrak{p}}}
\newcommand{\gP}{{\mathfrak{P}}}
\newcommand{\gq}{{\mathfrak{q}}}

\newcommand{\Abar}{{\bar A}}
\newcommand{\Ebar}{{\bar E}}
\newcommand{\Kbar}{{\bar K}}
\newcommand{\Pbar}{{\bar P}}
\newcommand{\Sbar}{{\bar S}}
\newcommand{\Tbar}{{\bar T}}
\newcommand{\abar}{{\bar a}}
\newcommand{\bbar}{{\bar b}}
\newcommand{\ybar}{{\bar y}}
\newcommand{\phibar}{{\bar\f}}
\newcommand{\mubar}{{\bar\mu}}

\newcommand{\ftilde}{{\tilde f}}

\newcommand{\Acal}{{\mathcal A}}
\newcommand{\Bcal}{{\mathcal B}}
\newcommand{\Ccal}{{\mathcal C}}
\newcommand{\Dcal}{{\mathcal D}}
\newcommand{\Ecal}{{\mathcal E}}
\newcommand{\Fcal}{{\mathcal F}}
\newcommand{\Gcal}{{\mathcal G}}
\newcommand{\Hcal}{{\mathcal H}}
\newcommand{\Ical}{{\mathcal I}}
\newcommand{\Jcal}{{\mathcal J}}
\newcommand{\Kcal}{{\mathcal K}}
\newcommand{\Lcal}{{\mathcal L}}
\newcommand{\Mcal}{{\mathcal M}}
\newcommand{\Ncal}{{\mathcal N}}
\newcommand{\Ocal}{{\mathcal O}}
\newcommand{\Pcal}{{\mathcal P}}
\newcommand{\Qcal}{{\mathcal Q}}
\newcommand{\Rcal}{{\mathcal R}}
\newcommand{\Scal}{{\mathcal S}}
\newcommand{\Tcal}{{\mathcal T}}
\newcommand{\Ucal}{{\mathcal U}}
\newcommand{\Vcal}{{\mathcal V}}
\newcommand{\Wcal}{{\mathcal W}}
\newcommand{\Xcal}{{\mathcal X}}
\newcommand{\Ycal}{{\mathcal Y}}
\newcommand{\Zcal}{{\mathcal Z}}

\renewcommand{\AA}{\mathbb{A}}
\newcommand{\BB}{\mathbb{B}}
\newcommand{\CC}{\mathbb{C}}
\newcommand{\FF}{\mathbb{F}}
\newcommand{\GG}{\mathbb{G}}
\newcommand{\HH}{\mathbb{H}}
\newcommand{\KK}{\mathbb{K}}
\newcommand{\NN}{\mathbb{N}}
\newcommand{\PP}{\mathbb{P}}
\newcommand{\QQ}{\mathbb{Q}}
\newcommand{\RR}{\mathbb{R}}
\newcommand{\ZZ}{\mathbb{Z}}

\newcommand{\bfa}{{\boldsymbol a}}
\newcommand{\bfb}{{\boldsymbol b}}
\newcommand{\bfc}{{\boldsymbol c}}
\newcommand{\bfe}{{\boldsymbol e}}
\newcommand{\bff}{{\boldsymbol f}}
\newcommand{\bfg}{{\boldsymbol g}}
\newcommand{\bfghat}{{\boldsymbol{\hat g}}}
\newcommand{\bfh}{{\boldsymbol h}}
\newcommand{\bfhhat}{{\boldsymbol{\hat h}}}
\newcommand{\bfp}{{\boldsymbol p}}
\newcommand{\bfm}{{\boldsymbol m}}
\newcommand{\bfr}{{\boldsymbol r}}
\newcommand{\bfs}{{\boldsymbol \sigma}}
\newcommand{\bft}{{\boldsymbol t}}
\newcommand{\bfu}{{\boldsymbol u}}
\newcommand{\bfv}{{\boldsymbol v}}
\newcommand{\bfw}{{\boldsymbol w}}
\newcommand{\bfx}{{\boldsymbol x}}
\newcommand{\bfy}{{\boldsymbol y}}
\newcommand{\bfz}{{\boldsymbol z}}
\newcommand{\bfA}{{\boldsymbol A}}
\newcommand{\bfB}{{\boldsymbol B}}
\newcommand{\bfC}{{\boldsymbol C}}
\newcommand{\bfD}{{\boldsymbol D}}
\newcommand{\bfE}{{\boldsymbol E}}
\newcommand{\bfF}{{\boldsymbol F}}
\newcommand{\bfG}{{\boldsymbol G}}
\newcommand{\bfH}{{\boldsymbol H}}
\newcommand{\bfI}{{\boldsymbol I}}
\newcommand{\bfM}{{\boldsymbol M}}
\newcommand{\bfU}{{\boldsymbol U}}
\newcommand{\bfzero}{{\boldsymbol{0}}}

\newcommand{\Amp}{\operatorname{Amp}}
\newcommand{\Aut}{\operatorname{Aut}}
\newcommand{\bad}{\textup{bad}}
\newcommand{\ds}{\displaystyle}
\newcommand{\Disc}{\operatorname{Disc}}
\newcommand{\dist}{\Delta}  
\newcommand{\Div}{\operatorname{Div}}
\renewcommand{\div}{{\textup{div}}}
\newcommand{\ehat}{{\hat\eta}}
\newcommand{\End}{\operatorname{End}}
\newcommand{\Eff}{\operatorname{Eff}}
\newcommand{\Family}{{\mathcal A}}  
\newcommand{\Fatou}{{\mathcal F}}
\newcommand{\Fbar}{{\bar{F}}}
\newcommand{\Fix}{\operatorname{Fix}}
\newcommand{\Gal}{\operatorname{Gal}}
\newcommand{\ghat}{{\hat g}}
\newcommand{\GL}{\operatorname{GL}}
\newcommand{\good}{\textup{good}}
\newcommand{\Hom}{\operatorname{Hom}}
\newcommand{\Index}{\operatorname{Index}}
\newcommand{\Image}{\operatorname{Image}}
\renewcommand{\Im}{\operatorname{Im}} 
\newcommand{\ii}{\sqrt{-1}} 
\newcommand{\interior}{\operatorname{int}}
\newcommand{\Julia}{{\mathcal J}}
\newcommand{\liftable}{{\textup{liftable}}}
\newcommand{\herm}{{\textup{herm}}}
\newcommand{\hhat}{{\hat h}}
\newcommand{\hhatlower}{\underline{\hat h}}
\newcommand{\hhatplus}{{\hat h^{\scriptscriptstyle+}}}
\newcommand{\hhatminus}{{\hat h^{\scriptscriptstyle-}}}
\newcommand{\hhatpm}{{\hat h^{\scriptscriptstyle\pm}}}
\newcommand{\hplus}{h^{\scriptscriptstyle+}}
\newcommand{\Ker}{{\operatorname{ker}}}
\newcommand{\Lift}{\operatorname{Lift}}
\newcommand{\MOD}[1]{~(\textup{mod}~#1)}
\newcommand{\Nef}{\operatorname{Nef}}
\newcommand{\Norm}{{\operatorname{\mathsf{N}}}}
\newcommand{\Nm}{\operatorname{Nm}}
\newcommand{\notdivide}{\nmid}
\newcommand{\normalsubgroup}{\triangleleft}
\newcommand{\Num}{{\operatorname{Num}}}
\newcommand{\NS}{{\operatorname{NS}}}
\newcommand{\odd}{{\operatorname{odd}}}
\newcommand{\onto}{\twoheadrightarrow}
\newcommand{\ord}{\operatorname{ord}}
\newcommand{\Orbit}{{\mathcal O}}
\newcommand{\PGL}{\operatorname{PGL}}
\newcommand{\Pic}{\operatorname{Pic}}
\newcommand{\Prob}{\operatorname{Prob}}
\newcommand{\psef}{\textup{psef}}
\newcommand{\Qbar}{{\bar{\QQ}}}
\newcommand{\qhat}{{\hat q}}
\newcommand{\rank}{\operatorname{rank}}
\renewcommand{\Re}{{\operatorname{Re}}} 
\newcommand{\Res}{{\operatorname{Res}}}
\newcommand{\Resultant}{\operatorname{Res}}
\newcommand{\rest}[2]{\left.{#1}\right\vert_{{#2}}}  
\renewcommand{\setminus}{\smallsetminus}
\newcommand{\Span}{\operatorname{Span}}
\newcommand{\Spec}{{\operatorname{Spec}}}
\newcommand{\Supp}{\operatorname{Supp}}
\newcommand{\sym}{{\textup{sym}}}
\newcommand{\Tate}[1]{\langle#1\rangle_E}
\newcommand{\TateA}[1]{\langle#1\rangle_A}
\newcommand{\TateX}[1]{\langle#1\rangle_X}
\newcommand{\tors}{{\textup{tors}}}
\newcommand{\Trace}{\operatorname{Trace}}
\newcommand{\Tr}{\operatorname{Tr}}
\newcommand{\UHP}{{\mathfrak{h}}}    
\newcommand{\Xhat}{{\hat X}}

\newcommand{\longhookrightarrow}{\lhook\joinrel\longrightarrow}
\newcommand{\longonto}{\relbar\joinrel\twoheadrightarrow}



\title[Dynamical Canonical Heights for Jordan Blocks\dots]
{Dynamical canonical heights for Jordan blocks,
arithmetic degrees of orbits, 
and nef canonical heights on abelian varieties}
\date{\today}

\author[Shu Kawaguchi and Joseph H. Silverman]
  {Shu Kawaguchi and Joseph H. Silverman}
\email{kawaguch@math.kyoto-u.ac.jp, jhs@math.brown.edu}
\address{Department of Mathematics, Faculty of Science, Kyoto University, 
Kyoto, 606-8502, Japan}
\address{Mathematics Department, Box 1917
         Brown University, Providence, RI 02912 USA}
\thanks{The first author's research supported by JSPS
grant-in-aid for young scientists (B) 24740015.
The second author's research supported by NSF DMS-0854755
and Simons Collaboration Grant \#241309.}

\begin{abstract}
Let $f:X\to X$ be an endomorphism of a normal projective variety
defined over a global field~$K$, and
let~$D_0,D_1,\ldots\in\Div(X)\otimes\CC$ form a Jordan block with
eigenvalue~$\l$ for the action of~$f^*$ on~$\Pic(X)\otimes\CC$.  We
construct associated canonical height functions~$\hhat_{D_k}$
satisfying Jordan transformation formulas $\hhat_{D_k}\circ f = \l
\hhat_{D_k} + \hhat_{D_{k-1}}$. As an application, we prove that for
every $x\in X(\Kbar)$, the arithmetic
degree~$\a_f(x)=\lim_{n\to\infty}h_X(f^n(x))^{1/n}$ of~$x$ exists, is
an algebraic integer, and takes on only finitely many values as~$x$
varies over~$X(\Kbar)$. Further, if $X$ is an abelian variety defined
over a number field and~$D$ is a nonzero nef divisor, we characterize
points satisfying $\hhat_D(x)=0$, and we use this characterization to
prove that if the $f$-orbit of $x$ is Zariski dense in~$X$,
then~$\a_f(x)$ is equal to the dynamical degree of~$f$.
\end{abstract}


\maketitle



\section*{Introduction}
\label{section:intro}
Let $K$ be a global field (as defined in
Section~\ref{section:defnotation}), let~$X/K$ be a normal projective
variety, let~$h_X:X(\Kbar)\to\RR$ be a Weil height associated to an
ample divisor, and let $\hplus_X=\max\{h_X,1\}$.  We study arithmetic
properties of the orbit of a point~$x\in X(\Kbar)$ under iteration of
a $K$-endomorphism $f:X\to X$.  We
recall~\cite{kawsilv:arithdegledyndeg,arxiv1111.5664} that the
\emph{$f$-arithmetic degree of~$x$} is the quantity
\begin{equation}
  \label{eqn:arithdegdef}
  \a_f(x) = \lim_{n\to\infty} \hplus_X\bigl(f^n(x)\bigr)^{1/n},
\end{equation}
if the limit exists.  The arithmetic degree of~$x$ provides a rough,
but useful, measure of the arithmetic complexity of the $f$-orbit
of~$x$. The following conjecture appears
in~\cite{kawsilv:arithdegledyndeg}.\footnote{More precisely, there 
is a version of Conjecture~\ref{conjecture:afxconj} 
in~\cite{kawsilv:arithdegledyndeg} that applies to dominant
rational maps $f:X\dashrightarrow X$. But when~$f$ is a morphism, we 
can drop the requirement that~$f$ be dominant.}

\begin{conjecture}
\label{conjecture:afxconj}
Let $f:X\to X$ be a $K$-endomorphism as described above.
\begin{parts}
\Part{(a)}
For every~$x\in X(\Kbar)$, the limit~\eqref{eqn:arithdegdef}
defining~$\a_f(x)$ converges.
\Part{(b)}
$\a_f(x)$ is an algebraic integer.
\Part{(c)}
The collection of arithmetic degrees
\[
 \bigl\{\a_f(x): x\in X(\Kbar)\bigr\}
\]
is a finite set.
\Part{(d)}
If the forward $f$-orbit of~$x$ is Zariski dense in~$X$,
then~$\a_f(x)$ is equal to~$\d_f$, the dynamical degree 
of~$f$.\footnote{Since $f$ is a morphism, its dynamical degree~$\d_f$
is simply the spectral radius of the map~$f^*$ that~$f$ induces
on~$\NS(X)\otimes\RR$.}
\end{parts}
\end{conjecture}

A consequence of one of the main theorems of this paper is the
following resolution of the first three parts of
Conjecture~\ref{conjecture:afxconj}.

\begin{theorem} 
\label{theorem:thmA} 
Let $f:X\to X$ be a $K$-endomorphism of a normal projective variety.
Then for all $x\in X(\Kbar)$, the limit~\eqref{eqn:arithdegdef}
defining the arithmetic degree~$\a_f(x)$ converges, and $
\bigl\{\a_f(x): x\in X(\Kbar)\bigr\}$ is a finite set of algebraic
integers, so parts~\textup{(a)},~\textup{(b)}, and~\textup{(c)} of
Conjecture~$\ref{conjecture:afxconj}$ are true.
\end{theorem}

The proof of Theroem~\ref{theorem:thmA} uses an extension of the
classical theory of dynamical canonical
heights~\cite{callsilv:htonvariety} to Jordan blocks, a result that is
of independent interest.  We recall that if a divisor
$D\in\Div(X)\otimes\CC$ satisfies a linear equivalence $f^*D\sim \l D$ for
some $\l\in\CC$ with $|\l|>1$, then the classical theory says that for
every~$x\in X(\Kbar)$, the Tate limit
\[
  \hhat_{f,D}(x)=\lim_{n\to\infty} \l^{-n}h_D\bigl(f^n(x)\bigr)
\]
exists, and the resulting function~$\hhat_{f,D}:X(\Kbar)\to\CC$
satisfies the following functional equation and normalization
condition:
\[
  \hhat_{f,D}\circ f=\l \hhat_{f,D}
  \qquad\text{and}\qquad
  \hhat_{f,D}=h_D+O(1). 
\]
We generalize this construction to allow a sequence of
divisors that form a Jordan block for the action of~$f^*$
on~$\Pic(X)\otimes\CC$.

\begin{theorem}
\label{theorem:thmB}
Let $X/K$ be a normal projective variety, let $f:X\to X$ be a
$K$-morphism, let $\l\in\CC$ with $|\l|>1$,
and let
\[
  D_0,D_1,D_2,\ldots \in\Div(X)\otimes\CC
\]
be divisors satisfying linear equivalences in Jordan block form,
\[
  \begin{array}{ccc@{\;}c@{\;}c@{\;}c@{\;}c@{\;}c}
    f^*D_0 & \sim & \l D_0 \\
    f^*D_1 & \sim & \phantom{\l}D_0 & + & \l D_1 \\
    f^*D_2 & \sim &&& \phantom{\l}D_1 & + & \l D_2 \\
    \vdots & &&&&\ddots && \ddots \\
  \end{array}
\]
Further, for each~$k$, let~$h_{D_k}$ be a Weil height function 
associated to the divisor~$D_k$.
\begin{parts}
\Part{(a)}
There are unique functions
\[
  \hhat_{D_0},\hhat_{D_1},\hhat_{D_2},\ldots :X(\Kbar)\longrightarrow\CC
\]
satisfying both the normalization conditions
\begin{equation}
  \label{eqn:hhatknomlztnintro}
  \hhat_{D_k} = h_{D_k} + O(1)
\end{equation}
and the functional equations
\begin{equation}
  \label{eqn:hhatkfuncteqnzzintro}
  \hhat_{D_k}\circ f = \l \hhat_{D_k} + \hhat_{D_{k-1}},
\end{equation}
where by convention we set $\hhat_{D_{-1}}=0$.
\Part{(b)}
The canonical height functions described in~\textup{(a)} satisfy
the recursively defined limit formulas
\begin{equation}
  \label{eqn:hhatrecursdef}
  \hhat_{D_k}(x) = \lim_{n\to\infty} 
   \left(\l^{-n}h_{D_k}\bigl(f^n(x)\bigr)
     - \sum_{i=1}^k \binom{n}{i}\l^{-i}\hhat_{D_{k-i}}(x)\right).
\end{equation}
\end{parts}
\end{theorem}

Part~(d) of Conjecture~\ref{conjecture:afxconj} appears to be more
difficult.  In Section~\ref{section:arithdegabelianvar} we prove
Conjecture~\ref{conjecture:afxconj}(d) in the case that~$X$ is an
abelian variety defined over a number field and~$f:X\to X$ is an
isogeny; see Corollary~\ref{corollary:Ofxdenseafxdf}.  The proof uses
the following result, also of independent interest, giving a
nontrivial generalization to nef divisors of the standard fact that
the canonical height on an abelian variety relative to an ample
divisor vanishes at precisely the points of finite order.

\begin{theorem}
\label{theorem:hADx0ax0}
Let $A/\Qbar$ be an abelian variety, let
$D\in\Div(A)\otimes\RR$ be a nonzero nef divisor, and
let~$\qhat_{A,D}$ be the quadratic part of the
canonical height on~$A$ with respect to~$D$.  Then there is a unique
abelian subvariety $B_D\subsetneq A$ such that
\[
  \bigl\{ x\in A(\Qbar) : \qhat_{A,D}(x)=0 \bigr\}
  = B_D(\Qbar) + A(\Qbar)_\tors.
\]
\end{theorem}

We make a number of remarks, after which we give a brief
section-by-section summary of the contents of the paper.

\begin{remark}
\label{remark:iteratefncleqn}
Iterating~\eqref{eqn:hhatkfuncteqnzzintro} gives a general functional
equation
\begin{equation}
  \label{eqn:hhatkfuncteqngeneral}
  \hhat_{D_k}\circ f^n = \sum_{i=0}^k \binom{n}{i}\l^{n-i}\hhat_{D_{k-i}},
\end{equation}
valid for all $n\ge0$; see Remark~\ref{remark:fnctindivhk}.
\end{remark}

\begin{remark}
The existence of the limit defining the arithmetic degree~$\a_f(x)$ is
significant, since it was observed in~\cite{kawsilv:arithdegledyndeg}
that the existence of~$\a_f(x)$ determines the asymptotic growth of
the orbit counting function via the formula
\[
  \lim_{T\to\infty}
    \frac{\#\bigl\{n\ge0 : h_X\bigl(f^n(x)\bigr)\le T\bigr\}}{\log T}
  = \frac{1}{\log\a_f(x)},
\]
where the limit is understood to be~$\infty$ if~$\a_f(x)=1$.
\end{remark}

\begin{remark}
An elaboration of the proof of Theorem~\ref{theorem:hhatk} can be used
to construct local canonical height functions (also sometimes called
Green functions) that are approrpiately normalized and satisfy
analogous local functional equations.  We refer the reader
to~\cite{callsilv:htonvariety} for a detailed description of the case
of a single eigendivisor $f^*D\sim \l D$.  But since the focus of this
paper is on global results, we have restricted attention to that
case.
\end{remark}

\begin{remark}
If the Jordan block \emph{linear equivalences} in
Theorem~\ref{theorem:thmB} are replaced by \emph{algebraic
  equivalences}, then a modification of the proof of
Theorem~\ref{theorem:thmB} gives canonical heights for points whose
(upper) arithmetic degrees are smaller than~$|\l|^2$.  These heights
satisfy the functional equation~\eqref{eqn:hhatkfuncteqnzzintro} and a
weak form of the normalization
condition~\eqref{eqn:hhatknomlztnintro}.  We briefly describe these
``algebraic equivalence canonical heights'' and their properties,
without proof, in Section~\ref{section:algequivhts}.
\end{remark}

\begin{remark}
If $D$ is an ample divisor on an abelian variety, then it is well known that
\begin{equation}
  \label{eqn:hAD0iffxinAtors}
  \bigl\{ x\in A(\Qbar) : \qhat_{A,D}(x)=0 \bigr\} = A(\Qbar)_\tors.
\end{equation}
Theorem~\ref{theorem:hADx0ax0} generalizes~\eqref{eqn:hAD0iffxinAtors}
to the case that~$D$ is only assumed to be nef.  In particular, we
note that if~$D\ne0$ is nef and~$A$ is geometrically simple,
then~\eqref{eqn:hAD0iffxinAtors} is still true, since the simplicity
of~$A$ forces~$B_D=0$. 
\par
More generally, if $f:X\to X$ is a morphism 
and~$D\in\Div(X)\otimes\RR$ satisfies an algebraic equivalence
$f^*D\equiv\l D$ for some $\l>\sqrt{\d_f}$, then the usual canonical
height limit $\hhat_{X,D,f}(x)=\lim \l^{-n}h_{X,D}\bigl(f^n(x)\bigr)$
exists; see Remark~\ref{remark:hhatiflgtdf12}
and~\cite[Theorem~5]{kawsilv:arithdegledyndeg}.
Further, if~$D$ is ample, then~$\hhat_{X,D}(x)=0$ implies that~$x$ is
$f$-preperiodic. But it seems to be a very interesting
question to describe the set \text{$\{x\in X(\Kbar):\hhat_{X,D}(x)=0\}$}
when~$D\ne0$ is nef, but not necessarily ample.
\end{remark}

\begin{remark}
In this paper we prove Conjecture~\ref{conjecture:afxconj} for abelian
varieties. In an earlier paper~\cite{arxiv1111.5664}, the second
author proved Conjecture~\ref{conjecture:afxconj} for the
torus~$\GG_m^N$. The proofs have some features in common, especially a
final, elementary linear algebra step in which one analyzes the kernel
of a linear transformation defined over a field that is larger than
the field of definition of the underlying vector space.  On the other
hand, the two cases use somewhat different tools to reduce to the
linear algebra step. For the toric case, one uses local canonical
height functions and Baker's theorem on linear forms in logarithms,
while for the abelian variety case, one uses the classification of nef
divisors and two fundamental global height formulas
(Propositions~\ref{proposition:axbybxay}
and~\ref{proposition:xyADxFDyAH}). It would be interesting to 
combine these two techniques to prove
Conjecture~\ref{conjecture:afxconj} for semi-abelian varieties.
\end{remark}

We conclude this introduction with a summary of the contents of the
paper.  Section~\ref{section:defnotation} begins by setting notation
and proving an elementary estimate for powers of a Jordan matrix.  In
Section~\ref{section:jordanalgeq} we prove Theorem~\ref{theorem:thmB},
which we restate slightly more generally as
Theorem~\ref{theorem:hhatk}.  This gives the existence of canonical
heights for Jordan blocks.  Section~\ref{section:arithdeg} contains
the proof of Theorem~\ref{theorem:thmA}, which verifies
Conjecture~\ref{conjecture:afxconj}(a,b,c) describing properties of
the arithmetic degree.  In Section~\ref{section:algequivhts} we state,
without proof, a version of Theorem~\ref{theorem:thmB} in which the
linear equivalences are replaced by algebraic equivalences.  In
Section~\ref{section:arithdegabelianvar} we study canonical heights
and arithmetic degrees on abelian varieties defined over number
fields.  In this setting, we classify points satisfying
$\qhat_{A,D}(x)=0$ for a nonzero nef divisor~$D$, and we prove that if
the orbit of a point is Zariski dense, then its arithmetic degree
takes on its maximum possible value, as predicted by
Conjecture~\ref{conjecture:afxconj}(d).  In addition, there are two
sections that work out examples illustrating the general theory.
Section~\ref{section:example} explicitly describes the 3-dimensional
Jordan block canonical heights associated to certain isogenies
$f:E^2\to E^2$, where~$E$ is a non-CM elliptic curve, and
Section~\ref{section:abeliansurfaceexample} examines the proof of
Theorem~\ref{theorem:hADx0ax0} for abelian varieties~$A$ whose
endomorphism algebra $\End(A)\otimes\QQ$ is  a real quadratic field.

\begin{acknowledgement}
The authors would like to thank Rob Benedetto for a helpful suggestion
that was instrumental in the eventual proof of
Theorem~\ref{theorem:thmA}, and Terry Loring for his quaternionic advice.
\end{acknowledgement}

\newpage

\section{Definitions and notation}
\label{section:defnotation}
In this section we set notation and give definitions that will be used
throughout the remainder of this paper (except for
Section~\ref{section:arithdegabelianvar} where we will take~$K$ to be
a number field). We begin by fixing:
\begin{notation}
\item[$K$]
a global field, which for the purposes of this paper will mean a
field~$K$ of characteristic~$0$, a fixed algebraic closure~$\Kbar$,
and a collection of absolute values on~$\Kbar$ such that there is a
well-defined theory of Weil height functions, as explained for example
in~\cite[Chapters~1--4]{lang:diophantinegeometry}.
\item[$X/K$]
a normal projective varitey.
\item[$f$]
a $K$-morphism $f:X\to X$.
\item[$h_X$]
a Weil height on $X(\Kbar)$ relative to an ample divisor.
\item[$\hplus_X$]
${}=\max\{h_X,1\}$.
\end{notation}

To ease notation, we use subscripts~$\QQ$,~$\RR$,
and~$\CC$ on abelian groups to indicate tensoring over~$\ZZ$ with the
indicated field. So for example, if~$X$ is a variety and~$A$ is an
abelian variety, then
\[
  \Div(X)_\CC=\Div(X)\otimes\CC,\quad
  \NS(X)_\RR=\NS(X)\otimes\RR,\quad
  A(\Qbar)_\QQ=A(\Qbar)\otimes\QQ.
\]

\begin{definition}
Let~$x\in X(\Kbar)$.  The \emph{arithmetic degree} of~$x$ is
the limit
\[
  \a_f(x) = \lim_{n\to\infty} \hplus_X\bigl(f^n(x)\bigr)^{1/n},
\]
if the limit exists. In any case, the \emph{upper and lower arithmetic
  degrees} of~$x$ are the quantities
\begin{equation}
  \label{eqn:alowaupdefs}
  \alower_f(x) = \liminf_{n\to\infty} \hplus_X\bigl(f^n(x)\bigr)^{1/n},
  \quad
  \aupper_f(x) = \limsup_{n\to\infty} \hplus_X\bigl(f^n(x)\bigr)^{1/n}.
\end{equation}
\end{definition}

\begin{definition}
For a given $\ell\ge0$ and $\l\in\CC$, we write $\L$ for the
$(\ell+1)$-dimensional lower Jordan block matrix
\[
  \L =
    \begin{pmatrix}
      \l & 0 & \cdots & 0 \\
       1 & \l & \cdots & 0 \\
       \vdots & & \ddots & \vdots \\
       0 & \cdots & 1 & \l \\
    \end{pmatrix}.
\]
\end{definition}

Since we will be studying relationships between height functions
associated to many different divisors, it is convenient to use 
vector-valued height functions; cf.\ \cite{MR2011389}.

\begin{definition}
For divisors~$D_0,\ldots,D_\ell\in\Div(X)_\CC$ and associated height
functions~$h_{D_0},\ldots,h_{D_\ell}$, we define a (column)
vector-valued height function
\[
  \bfh_\bfD : X(\Kbar) \longrightarrow \CC^{\ell+1},
  \qquad
  \bfh_\bfD(x) = {}^t\bigl(h_{D_0}(x),\ldots,h_{D_\ell}(x)\bigr).
\]
If the divisors and heights have been fixed, to ease notation we write
\[
  h_k=h_{D_k}\quad\text{for $0\le k\le\ell$.}
\]
\end{definition}

\begin{definition}
We use \text{$\|\,\cdot\,\|$} to denote the sup norm of a vector or a
matrix, i.e., for vectors $\bfa=(a_i)$ and matrices $B=(b_{ij})$ with
complex coordinates,
\[
  \|\bfa\|=\max|a_i| \quad\text{and}\quad \|B\|=\max |b_{ij}|.
\]
We will frequently use the elementary triangle inequality estimate
\[
  \|B\bfa\| \le \dim(B)\cdot \|B\|\cdot \|\bfa\|.
\]
\end{definition}

We prove two elementary estimates about the powers of a Jordan matrix.

\begin{lemma}
\label{lemma:jordanmatrixests}
\textup{(a)}
For all $n\ge1$,
\[
  \|\L^n\| \le n^\ell \max\bigl\{|\l|,1\bigr\}^n.
\]
\par\noindent\textup{(b)}\enspace
If  $\bfv\in\CC^{\ell+1}$ is a nonzero vector, then
\[
  \lim_{n\to\infty} \|\L^n\bfv\|^{1/n} = |\l|.
\]
\end{lemma}
\begin{proof}
We write $\L=\l I+N$ with $N$ nilpotent. Then $N^{\ell+1}=0$, so
\[
  \L^n = (\l I + N)^n
  = \sum_{i=0}^\ell \binom{n}{i} \l^{n-i}N^i.
\]
\par\noindent(a)\enspace
Using the nilpotent form of~$N$, we see that the
$ij^{\text{th}}$-entry of~$\L^n$ is
\[
  \L^n_{ij} = \binom{n}{i-j} \l^{n-(i-j)},
\]
where we set $\binom{n}{k}=0$ if $k<0$, and of
course~$\binom{n}{k}=0$ if $k>n$.  Hence
\[
  \|\L^n\| =  \max_{0\le k\le\min\{\ell,n\}} \binom{n}{k} |\l|^{n-k} 
  \le n^\ell \max\bigl\{|\l|,1\bigr\}^n.
\]
\par\noindent(b)\enspace
If $\l=0$, then $\L^{\ell+1}=0$, so the result is trivially true.  We
assume now that~$\l\ne0$.  We write $\bfv={}^t(v_0,\ldots,v_\ell)$.
Then the $i^{\text{th}}$-coordinate of~$\L^n\bfv$ satisfies
\[
  (\L^n\bfv)_i
  = \sum_{j=0}^\ell \L^n_{ij} v_j
  = \sum_{j=0}^\ell \binom{n}{i-j} \l^{n-(i-j)}  v_j
  = O(n^{\ell}\l^n).
\]
This holds for all~$i$, so we have
\begin{equation}
  \label{eqn:limsuplel}
  \limsup_{n\to\infty} \|\L^n\bfv\|^{1/n}
  \le \limsup_{n\to\infty} \bigl\|O(n^{\ell}\l^n)\bigr\|^{1/n}
  \le |\l|.
\end{equation}
\par
Next, since $\bfv\ne\bfzero$, there is an index $0\le k\le\ell$ such
that
\[
  v_0 = \cdots = v_{k-1} = 0
  \quad\text{and}\quad
  v_k \ne 0.
\]
Hence the $k$'th coordinate of~$\L^n\bfv$ has the form
\[
  (\L^n\bfv)_k 
  = \sum_{j=1}^\ell\L_{kj}^nv_j
  = \sum_{j=0}^\ell \binom{n}{k-j} \l^{n-(k-j)}  v_j
  = \l^nv_k,
\]
where the last equality follows from the facts that $v_j=0$ for $j<k$
and $\binom{n}{k-j}=0$ for $j>k$. This gives
\begin{equation}
  \label{eqn:liminfgel}
  \liminf_{n\to\infty} \|\L^n\bfv\|^{1/n}
  \ge \liminf_{n\to\infty} \bigl|(\L^n\bfv)_k\bigr|^{1/n}
  =  \liminf_{n\to\infty} |\l^nv_k|^{1/n}
  = |\l|,
\end{equation}
where the last equality follows from the fact that $v_k\ne0$.
Combining~\eqref{eqn:limsuplel} and~\eqref{eqn:liminfgel} gives
\[
  \lim_{n\to\infty} \|\L^n\bfv\|^{1/n} = |\l|,
\]
which completes the proof of Lemma~\ref{lemma:jordanmatrixests}.
\end{proof}

\section{Jordan block canonical heights}
\label{section:jordanalgeq}

In this section we construct canonical heights associated to Jordan
blocks in~$\Pic(X)_\CC$. These exist for eigenvalues satisfying
$|\l|>1$.  We also give an estimate that holds for all~$\l$.

\begin{theorem}
\label{theorem:hhatk}
Let $\l\in\CC$, and let $D_0,D_1,\ldots,D_\ell\in\Div(X)_\CC$ be
divisors that form a Jordan block in~$\Pic(X)_\CC$.
\begin{equation}
  \label{eqn:jordanblocks}
  f^*D_0\sim \l D_0,\quad
  f^*D_1\sim D_0 + \l D_1,\quad\ldots\quad
  f^*D_\ell\sim D_{\ell-1} + \l D_\ell.
\end{equation}
\vspace{-12pt}
\begin{parts}
\Part{(a)}
There is a constant $C=C(D_0,\ldots,D_\ell,\l)$ such that
\begin{align*}
  \bigl\|\bfh_\bfD\bigl(f^n(x)\bigr)\bigr\|
  \le  C n^\ell \max \bigl\{ |\l|, & 1 \bigr\}^n
          \cdot \bigl(\bigl\| \bfh_\bfD(x) \bigr\|+1\bigr) \\*
  &\text{for all $x\in X(\Kbar)$ and all $n\ge0$.}
\end{align*}
\Part{(b)}
If $|\l|>1$, then there is a unique function
$\bfhhat_\bfD:X(\Kbar)\to\CC^{\ell+1}$ satisfying the functional
equation
\begin{equation}
  \label{eqn:hhatkfuncteqnzz}
  \bfhhat_\bfD\circ f = \L \bfhhat_\bfD
\end{equation}
and the normalization condition
\begin{equation}
  \label{eqn:bfhhatnormcond}
  \bfhhat_\bfD = \bfh_\bfD + O(1).
\end{equation}
\Part{(c)}
The coordinate functions of the canonical height function
\[
  \bfhhat_\bfD={}^t(\hhat_0,\hhat_1,\ldots,\hhat_\ell)
\]
described in~\textup{(b)} satisfy the limit recursion relations
\begin{equation}
  \label{eqn:hhatkrecurxx}
  \hhat_k(x) = \lim_{n\to\infty} 
   \left(\l^{-n}h_k\bigl(f^n(x)\bigr)
     - \sum_{i=1}^k \binom{n}{i}\l^{-i}\hhat_{k-i}(x)\right).
\end{equation}
\end{parts}
\end{theorem}

\begin{remark}
We note that in Theorem~\ref{theorem:hhatk}, if
$D_0,\ldots,D_\ell\in\Div(X)_\RR$ and $\l\in\RR$, then the associated
canonical height functions take values in~$\RR$, i.e.,
$\bfhhat_\bfD:X(\Kbar)\to\RR^{\ell+1}$.
\end{remark}

\begin{remark}
We also note that the proof of Theorem~\ref{theorem:hhatk} does not use
the assumption that the global field~$K$ has characteristic~$0$, so it
in fact holds for any field on which there is a theory of Weil height
functions.
\end{remark}

\begin{remark}
\label{remark:fnctindivhk}
The functional equation for the individual~$\hhat_k\circ f^n$
mentioned in Remark~\ref{remark:iteratefncleqn} follows immediately
from the functional equation~\eqref{eqn:hhatkfuncteqnzz}.  To see
this, we iterate~\eqref{eqn:hhatkfuncteqnzz} to obtain
$\bfhhat_\bfD\circ f^n = \L^n \bfhhat_\bfD$, expand $\L^n=(\l I+N)^n$
using the binomial theorem, and apply the fact (already used in the
proof of Lemma~\ref{lemma:jordanmatrixests} (a)) that
$\L_{ij}^n=\binom{n}{i-j}\l^{n-(i-j)}$.
\end{remark}

\begin{proof}[Proof of Theorem \textup{\ref{theorem:hhatk}}]
We define a vector-valued ``error function''
\[
  \bfE_\bfD:X(\Kbar)\to\CC^{\ell+1},\qquad
  \bfE_\bfD = \bfh_\bfD\circ f - \L \bfh_\bfD.
\]
The assumption~\eqref{eqn:jordanblocks} that~$D_0,\ldots,D_\ell$
form a Jordan block says that each coordinate function of~$\bfE_\bfD$
is a Weil height function with respect to a divisor that is linearly
equivalent to zero. A standard property of Weil height
functions~\cite[Theorem~B.3.2(d)]{hindrysilverman:diophantinegeometry}
then implies that there is a constant~$C_1$ such that
\begin{equation}
  \label{eqn:hDfLhDestlin}
  \bigl\| \bfE_\bfD(x) \bigr\| \le C_1
     \quad\text{for all~$x\in X(\Kbar)$.} 
\end{equation}
\par
We now begin the proof of~(a).  For $N\ge1$ we consider the
telescoping sum
\begin{align*}
  \bfh_\bfD\circ f^N
   &= \L^N \bfh_\bfD + \sum_{n=0}^{N-1} 
          \L^{N-n-1}\bigl(\bfh_\bfD\circ f^{n+1}-\L\bfh_\bfD\circ f^n\bigr)\\*
   &= \L^N \bfh_\bfD + \sum_{n=0}^{N-1} \L^{N-n-1}\bfE_\bfD\circ f^n.
\end{align*}
To ease notation, we let
\[
  |\lplus| = \max\bigl\{|\l|,1\bigr\},
\]
so Lemma~\ref{lemma:jordanmatrixests}(a) says that $\|\L^n\|\le
n^\ell|\lplus|^n$.  For $x\in X(\Kbar)$ we compute
\begin{align*}
  \bigl\| & \bfh_\bfD\bigl(f^N(x)\bigr) \bigr\| \\*
  &\le \bigl\| \L^N \bfh_\bfD(x) \bigr\|
    + \sum_{n=0}^{N-1} \bigl\| \L^{N-n-1}\bfE_\bfD\bigl(f^n(x)\bigr) \bigr\| \\
  &\le (\ell+1) \| \L^N \| \cdot \bigl\| \bfh_\bfD(x) \bigr\| 
    + \sum_{n=0}^{N-1} (\ell+1) \| \L^{N-n-1}\| 
          \cdot \bigl\|\bfE_\bfD\bigl(f^n(x)\bigr) \bigr\| \\
  &\le (\ell+1) N^\ell |\lplus|^N
          \cdot \bigl\| \bfh_\bfD(x) \bigr\|\\*
  &\hspace{2em}{}
    + \smash[b]{ (\ell+1) \sum_{n=0}^{N-1}  (N-n-1)^\ell 
               |\lplus|^{N-n-1}
          \cdot \bigl\|\bfE_\bfD\bigl(f^n(x)\bigr) \bigr\| } \\*
  &\omit\hfill\text{from Lemma~\ref{lemma:jordanmatrixests}(a),} \\
  &\le (\ell+1) N^\ell |\lplus|^N 
         \cdot \bigl\| \bfh_\bfD(x) \bigr\|\\*
  &\hspace{2em}{}
    + (\ell+1) \sum_{n=0}^{N-1}  (N-n-1)^\ell |\lplus|^{N-n-1} C_1
    \quad\text{from \eqref{eqn:hDfLhDestlin},} \\*
  &\le (\ell+1) N^\ell |\lplus|^N
          \cdot \bigl\| \bfh_\bfD(x) \bigr\|
      + C_1 (\ell+1) N^\ell |\lplus|^N .
\end{align*}
This completes the proof of~(a).
\par
We next show that there is at most one vector-valued height
function~$\bfhhat_\bfD$ satisfying the conditions given in~(b).  So we
suppose that~$\bfhhat_\bfD'$ is another such function, we let
$\bfghat_\bfD=\bfhhat_\bfD-\bfhhat_\bfD'$, and we prove by
contradiction that~$\bfghat_\bfD=0$.  So we suppose that there is a
point $x\in X(\Kbar)$ such that~$\bfghat_\bfD(x)\ne0$.
\par
Taking the difference of the functional equations for~$\bfhhat_\bfD$
and~$\bfhhat_\bfD'$ gives a functional equation 
\begin{equation}
  \label{eqn:gfnceq}
  \bfghat_\bfD\circ f=\L \bfghat_\bfD
\end{equation}
for $\bfghat_\bfD$.  This allows us to compute
\begin{align*}   
  |\l|
  &= \limsup_{n\to\infty} \bigl\| \L^n \bfghat_\bfD(x) \bigr\|^{1/n}
    \quad\text{from Lemma~\ref{lemma:jordanmatrixests}(b), 
          since  $\bfghat_\bfD(x)\ne0$,} \\*
  &=\limsup_{n\to\infty} \bigl\| \bfghat_\bfD\bigl(f^n(x)\bigr) \bigr\|^{1/n}
    \quad\text{from \eqref{eqn:gfnceq},} \\
  &=
  \limsup_{n\to\infty} \bigl\| \bfhhat_\bfD\bigl(f^n(x)\bigr)
       - \bfhhat_\bfD'\bigl(f^n(x)\bigr) \bigr\|^{1/n} 
    \quad\text{definition of $\bfghat_\bfD$,}\\
  &\le   \limsup_{n\to\infty} \Bigl(
   \bigl\|\bfhhat_\bfD\bigl(f^n(x)\bigr) - \bfh_\bfD\bigl(f^n(x)\bigr)\bigr\| 
         \hspace{2em}\\*
  &\omit\hfill$\displaystyle
  +\bigl\|\bfhhat_\bfD'\bigl(f^n(x)\bigr) - \bfh_\bfD\bigl(f^n(x)\bigr)\bigr\|
           \Bigr)^{1/n}$ \\*
  &\le  1
    \quad\text{from \eqref{eqn:bfhhatnormcond}.}
\end{align*}
Hence $|\l|\le1$, which contradicts the assumption in~(b) that
$|\l|>1$. This contradiction completes the proof
that~$\bfghat_\bfD=0$, which shows that~$\bfhhat_\bfD$ is uniquely
determined by~\eqref{eqn:hhatkfuncteqnzz}
and~\eqref{eqn:bfhhatnormcond}.
\par
The assumption that~$|\l|>1$ allows us to estimate the norm of the
negative powers of~$\L$ as follows:
\begin{align}
  \label{eqn:absLnegn}
  \|\L^{-n}\|
  &= \bigl\| (\l I+N)^{-n} \bigr\| \notag \\*
  &= \left\| \sum_{i=0}^{\ell} \binom{-n}{i}\l^{-n-i}N^i \right\|
     \quad\text{since $N^{\ell+1}=0$,} \notag \\
  &\le |\l|^{-n} \cdot (\ell+1)
          \cdot \max_{0\le i\le\ell} \left|\binom{-n}{i}\right| \notag \\*
  &\le C_2 n^\ell |\l|^{-n},
\end{align}
where the constant~$C_2$ depends on~$\ell$ and~$\l$, but does not
depend on~$n$.  

\begin{claim}
\label{claim:absconvhhat}
For all $x\in X(\Kbar)$, the vector-valued series
\begin{equation}
  \label{eqn:defbfhhat}
  \bfhhat_\bfD(x) := \bfh_\bfD(x)
       + \sum_{n=0}^\infty \L^{-n-1} \bfE_\bfD\bigl(f^n(x)\bigr)
\end{equation}
is absolutely convergent and defines a vector-valued height function
\[
  \bfhhat_\bfD:X(\Kbar)\to\CC^{\ell+1}
  \quad\text{satisfying}\quad
  \bigl\|\bfhhat_\bfD(x) - \bfh_\bfD(x)\bigr\| \le C
\]
for a constant~$C$ that is independent of~$x$.
\end{claim}
\begin{proof}
We compute
\begin{align*}
  \sum_{n=0}^\infty \bigl\| \L^{-n-1} \bfE_\bfD\bigl(f^n(x)\bigr) \bigr\|
  &\le (\ell+1) \sum_{n=0}^\infty \|\L^{-n-1}\| 
       \cdot \bigl\|\bfE_\bfD\bigl(f^n(x)\bigr)\bigr\| \\
  &\le C_3 \sum_{n=0}^\infty n^\ell |\l|^{-n}
           \bigl\|\bfE_\bfD\bigl(f^n(x)\bigr)\bigr\| 
    \quad\text{from \eqref{eqn:absLnegn},} \\
  &\le C_3C_1 \sum_{n=0}^\infty n^\ell |\l|^{-n} 
    \quad\text{from \eqref{eqn:hDfLhDestlin},} \\
  &  \le C_4
    \quad\text{since~$|\l|>1$ by assumption.}
\end{align*}
This shows that the series appearing in~\eqref{eqn:defbfhhat} is
absolutely convergent, while simultaneously giving the desired upper
bound for $\bigl\|\bfhhat_\bfD(x) - \bfh_\bfD(x)\bigr\|$, which
completes the proof of the claim.
\end{proof}

Claim~\ref{claim:absconvhhat} gives us a well-defined
function~$\bfhhat_\bfD:X(\Kbar)\to\CC^{\ell+1}$ that satisfies the
normalization condition~\eqref{eqn:bfhhatnormcond}. It remains to
prove that~$\bfhhat_\bfD$ satisfies the functional
equation~\eqref{eqn:hhatkfuncteqnzz}.  Since we know that the series
defining~$\bfhhat_\bfD$ is absolutely convergent, the proof is a
formal calculation using the definitions of~$\bfhhat_\bfD$
and~$\bfE_\bfD$. Thus
\begin{align*}
  \bfhhat_\bfD\circ f
  &= \bfh_\bfD\circ f + \sum_{n=0}^\infty \L^{-n-1} \bfE_\bfD\circ f^{n+1} \\
  &= \bfh_\bfD\circ f + \sum_{n=1}^\infty \L^{-n} \bfE_\bfD\circ f^{n} \\
  &= \bfh_\bfD\circ f - \bfE_\bfD + \sum_{n=0}^\infty \L^{-n} \bfE_\bfD\circ f^{n} \\
  &= \L \bfh_\bfD + \L  \sum_{n=0}^\infty \L^{-n-1} \bfE_\bfD\circ f^{n} \\
  &= \L \bfhhat_\bfD.
\end{align*}
\par\noindent(c)\enspace
We have already proven that the canonical height
function~$\bfhhat_\bfD$ exists and satisfies a normalization condition
and a functional equation. In particular, the normalization condition
implies that
\begin{equation}
  \label{eqn:hfnhfnoln}
  \bfhhat_\bfD\circ f^n = \bfh_\bfD\circ f^n + O(1).
\end{equation}
We compute
\begin{align*}
  \bfhhat_\bfD
  &= \bfhhat_\bfD + \l^{-n}(\bfhhat_\bfD\circ f^n - \L^n \bfhhat_\bfD) 
   \quad\text{from the functional equation,} \\
  &= \l^{-n}\bfh_\bfD\circ f^n - \bigl((\l^{-1}\L)^n - I\bigr)\bfhhat_\bfD
          + \l^{-n}(\bfhhat_\bfD\circ f^n - \bfh_\bfD\circ f^n) \\
  &= \l^{-n}\bfh_\bfD\circ f^n - \bigl((\l^{-1}\L)^n - I\bigr)\bfhhat_\bfD 
     + O(\l^{-n})
    \quad\text{from \eqref{eqn:hfnhfnoln},} \\
  &= \l^{-n}\bfh_\bfD\circ f^n - \bigl((I+\l^{-1}N)^n - I\bigr)\bfhhat_\bfD 
       + O(\l^{-n}) \\
  &\omit\hfill \quad\text{writing $\L=\l I+N$,} \\
  &= \l^{-n}\bfh_\bfD\circ f^n -
      \smash[t]{
         \sum_{i=1}^\ell \binom{n}{i}\l^{-i}N^i\bfhhat_\bfD + O(\l^{-n})
       }
    \quad\text{since $N^{\ell+1}=0$.}
\end{align*}
Evaluating at~$x$ and letting~$n\to\infty$ yields
\begin{equation}
  \label{eqn:hDxnlnbnm}
  \bfhhat_\bfD(x)
  = \lim_{n\to\infty}
  \left( \l^{-n}\bfh_\bfD\bigl(f^n(x)\bigr)
        - \sum_{i=1}^\ell \binom{n}{i}\l^{-i}N^i\bfhhat_\bfD(x) \right).
\end{equation}
Multiplication by the matrix $N$ is the right-shift operator
\[
  {}^t(a_0,\ldots,a_\ell)\longmapsto {}^t(0,a_0,\ldots,a_{\ell-1}),
\]
so the $k^{\text{th}}$-coordinate of $N^i\bfhhat_\bfD(x)$ is
$\hhat_{k-i}$. (By convention we set $\hhat_j=0$ if $j<0$.)  This
shows that the vector formula~\eqref{eqn:hDxnlnbnm} is a succinct way
of writing the formulas~\eqref{eqn:hhatkrecurxx} that we are trying to
prove.
\end{proof}

\section{An example of a Jordan block canonical height}
\label{section:example}

In this section we illustrate Theorem~\ref{section:jordanalgeq}
for~$X=E^2$, where~$E/K$ is a non-CM elliptic curve. In this case, the
N\'eron--Severi group of~$X$ is generated by the three divisors
\[
  H_1=(O)\times E,\quad
  H_2=E\times(O),\quad\text{and}\quad
  \D=\bigl\{(x,x):x\in E\bigr\}.
\]
It is more convenient to take as our basis for~$\NS(X)_\QQ$
the three divisors
\[
  H_1,\quad H_2,\quad\text{and}\quad
  H_3 = H_1+H_2-\D.
\]
We consider an endomorphism of the form
\[
  f:X\longrightarrow X,\qquad
  f(x,y)=(ax+by,ay)
\]
with $a,b\in\ZZ$ satisfying $|a|\ge2$ and~$b\ne0$. An elementary
intersection theory calculation yields the formulas
\begin{align*}
  f^*H_1 &\equiv a^2H_1 + b^2H_2 + abH_3, \\
  f^*H_2 &\equiv a^2H_2,\\
  f^*H_3 &\equiv 2abH_2+a^2H_3,
\end{align*}
from which it is a linear algebra exercise to construct divisors
\[
  D_0 = 4a^3b^2 H_2,\qquad
  D_1 = 2a^2b H_3,\qquad
  D_2 = 2aH_1 - bH_3,
\]
satisfying
\[
  f^*D_0 \equiv a^2D_0,\qquad
  f^*D_1 \equiv D_0+a^2D_1,\qquad
  f^*D_2 \equiv D_1+a^2D_2.
\]
Thus~$D_0,D_1,D_2$ form a Jordan block basis for~$\NS(X)_\RR$, and
their associated Jordan block canonical heights satisfy
\[
  \hhat_{D_0}\circ f = a^2\hhat_{D_0},\quad
  \hhat_{D_1}\circ f = \hhat_{D_0}+a^2\hhat_{D_1},\quad
  \hhat_{D_2}\circ f = \hhat_{D_1}+a^2\hhat_{D_2}.
\]
A short calculation shows that $\hhat_{D_0}$, $\hhat_{D_1}$,
$\hhat_{D_2}$ may be written in terms of the canonical height pairing
on~$E$ as follows:
\begin{align*}
  \hhat_{D_0}(x,y) &= 4a^3b^2 \Tate{y,y}, \\
  \hhat_{D_1}(x,y) &= 2a^2b \Tate{x,y}, \\
  \hhat_{D_2}(x,y) &= \Tate{x,2ax-by}.
\end{align*}

\section{An application to arithmetic degrees} 
\label{section:arithdeg}

In this section we prove Theorem~\ref{theorem:thmA}, which says that
for morphisms $f:X\to X$, the arithmetic degree~$\a_f(x)$ of a point
$x\in X(\Kbar)$ exists, is an algebraic integer, and takes on only
finitely many values as~$x$ ranges over~$X(\Kbar)$.

We recall that the lower and upper arithmetic
degrees~\eqref{eqn:alowaupdefs} are defined, respectively, to be the
liminf and limsup of $\hplus_X\bigl(f^n(x)\bigr)^{1/n}$.  It is proven
in~\cite[Proposition~13]{kawsilv:arithdegledyndeg} that the values
of~$\alower_f(x)$ and~$\aupper_f(x)$ do not depend on the choice of
the ample height function~$h_X$. 

Before beginning the proof of Theorem~\ref{theorem:thmA}, we note that
the equivalence class of Weil height functions associated to a divisor
$D\in\Div(X)_\CC$ consists of complex valued functions
\[
  h_D:X(\Kbar)\longrightarrow\CC
\]
obtained by writing $D$ as a linear combination
\[
  D = c_1D_1+\cdots+c_tD_t
  \quad\text{with $c_i\in\CC$ and $D_i\in\Div(X)$}
\]
and setting
\[
  h_D = c_1h_{D_1}+\cdots+c_th_{D_t}+O(1).
\]
On the other hand, we note that the concept of ample divisor only
makes sense in~$\Div(X)_\RR$, not in~$\Div(X)_\CC$.  Despite this
potential problem, the following helpful lemma uses divisors
in~$\Div(X)_\CC$ to say something about arithmetic degrees, whose
definition requires using a height associated to an ample divisor.

\begin{lemma}
\label{lemma:alowergehDfn}
Let $f:X\to X$ be a morphism, let $D\in\Div(X)_\CC$ by any divisor,
and let  $x\in X(\Kbar)$. Then
\begin{equation}
  \label{eqn:alowbeliminfhD}
   \alower_f(x) \ge \liminf_{n\to\infty} \bigl|h_D\bigl(f^n(x)\bigr)\bigr|^{1/n}.
\end{equation}
\end{lemma}
\begin{proof}
If
\[
  \lim_{n\to\infty} \bigl|h_D\bigl(f^n(x)\bigr)\bigr| \ne \infty,
\]
then the right-hand side of~\eqref{eqn:alowbeliminfhD} is less than or equal
to~$1$, so~\eqref{eqn:alowbeliminfhD} is automatically true, since
the definition of~$\alower_f(x)$ ensures that~$\alower_f(x)\ge1$. We may
thus assume that $\bigl|h_D\bigl(f^n(x)\bigr)\bigr|\to\infty$.
\par
Any two nontrivial norms on~$\CC$ are related by $|z|_1\asymp|z|_2$,
so the right-hand side of~\eqref{eqn:alowbeliminfhD} is independent of
the chosen norm on~$\CC$. We will use the norm
$|a+bi|=\max\bigl\{|a|,|b|\bigr\}$. We write
\[
  D = D_1 + iD_2\quad\text{with $D_1,D_2\in\Div(X)_\RR$.}
\]
We fix an ample divisor $H\in\Div(X)$. Then there is an $\e>0$ so that
the four divisors
\[
  H+\e D_1,\quad H-\e D_1,\quad H+\e D_2,\quad H-\e D_2
\]
are all in the ample cone. Since~$H$ is ample, we may choose a
height function~$h_H$ for~$H$ satisfying $h_H\ge0$.
\par 
Using standard properties of height functions, we estimate
\begin{align*}
  \max\left\{ h_{H+\e D_1}, h_{H-\e D_1} \right\}
  &=  h_H + \e \max\left\{ h_{D_1}, -h_{D_1} \right\} + O(1) \\
  &=  h_H + \e \left| h_{D_1} \right| + O(1) \\
  &\ge  \e \left| h_{D_1} \right| + O(1),
\end{align*}
where the last line follows because $h_H\ge0$. Replacing~$D_1$ with~$D_2$
gives an analogous inequality, so we find that
\begin{align}
  \label{eqn:HpmeDigehD}
  \max\{ h_{H+\e D_1},h_{H-\e D_1}, h_{H+\e D_2}, h_{H-\e D_2}\} 
  &\ge \e \max \left\{ \left| h_{D_1} \right|, \left| h_{D_2} \right| \right\}
     + O(1) \notag\\
  &= \e |h_D| + O(1).
\end{align}
We now evaluate~\eqref{eqn:HpmeDigehD} at~$f^n(x)$, take the
$n^{\text{th}}$-root, and take the liminf as $n\to\infty$. Since the
divisors $H\pm\e D_i$ are all ample, we can use the fact
that~$\alower_f(x)$ may be computed using any ample
divisor~\cite[Proposition~13]{kawsilv:arithdegledyndeg} to deduce that
the left-hand side of~\eqref{eqn:HpmeDigehD} goes
to~$\alower_f(x)$.\footnote{The main theorems
  in~\cite{kawsilv:arithdegledyndeg} assume that~$X$ is smooth and~$f$
  is dominant, but neither of these assumptions is used in the proof
  of~\cite[Proposition~13]{kawsilv:arithdegledyndeg}.}  On the other
hand, since $\e>0$ and $\bigl|h_D\bigl(f^n(x)\bigr)|\to\infty$, we see
that the right-hand side of~\eqref{eqn:HpmeDigehD} is equal to
$\liminf \bigl|h_D\bigl(f^n(x)\bigr)|^{1/n}$.
\end{proof}

\begin{lemma}
\label{lemma:PffstarPicX}
Let $f:X\to X$ be a morphism. Then there is a monic integral polynomial
$P_f(t)\in\ZZ[t]$ with the property that $P_f(f^*)$ annihilates~$\Pic(X)$.
\end{lemma}
\begin{proof}
The Picard group of~$X$ fits into an exact sequence
\[
  0 \longrightarrow A \longrightarrow \Pic(X)
    \longrightarrow \NS(X) \longrightarrow 0,
\]
where~$A=\Pic^0(X)$ is an abelian variety, and where
the N\'eron--Severi group~$\NS(X)$ is a finitely generated abelian
group by the theorem of the base~\cite[Chapter~6,
  Theorem~6.1]{lang:diophantinegeometry}.  In paricular,
since~$\NS(X)$ is finitely generated, there is a monic
$Q_f(t)\in\ZZ[t]$ such that $Q_f(f^*)$ annihilates~$\NS(X)$.
\par
The map $f^*:\Pic(X)\to\Pic(X)$ maps~$\Pic^0(X)$ to itself, and the
resulting map is an endomorphism, which we denote by~$\f_f:A\to A$. 
Let~$R_f(t)\in\ZZ[t]$ to be the characteristic polynomial of~$\f_f$
acting on the Tate module~$T_\ell(A)$. Then~$R_f(\f_f)\in\End(A)$
annihilates all of the $\ell$-power torsion of~$A$, so~$R_f(\f_f)=0$.
Setting $P_f(t)=R_f(t)Q_f(t)$ then gives a monic integral polynomial
satisfying $P_f(f^*)(D)\sim 0$ for all $D\in\Pic(X)$.
\end{proof}

\begin{proof}[Proof of Theorem \textup{\ref{theorem:thmA}}]
Let~$P_f(t)\in\ZZ[t]$ be the monic polynomial from
Lemma~\ref{lemma:PffstarPicX} having the property
that~$P_f(f^*)(D)\sim0$ for all~$D\in\Pic(X)$, and let~$d=\deg(P_f)$.
We fix an ample divisor~$H\in\Div(X)$, and we let
\[
  V = \Span_\QQ\bigl(H,\,f^*H,\,(f^*)^2H,\ldots,(f^*)^{d-1}H\bigr)
  \subset \Pic(X)_\QQ.
\]
Then the fact that $P_f(f^*)(H)\sim0$ implies that~$V$ is
an~$f^*$-invariant subspace of~$\Pic(X)_\QQ$.
We let $\rho=\dim(V)$.
\par
Extending scalars to~$\CC$, we choose divisors
$E_1,\ldots,E_\rho\in\Div(X)_\CC$ whose divisor classes
in~$\Pic(X)_\CC$ form a $\CC$-basis for $V_\CC$ such that the
associated matrix of~$f^*|_V$ is in Jordan normal form. Thus for each
$1\le i\le\rho$, we have either
\[
  f^*E_i \sim \l_i E_i
  \quad\text{or}\quad 
  f^*E_i \sim \l_i E_i + E_{i-1},
\]
where by convention we set $E_0=0$. 
\par
Relabeling the divisors, we may assume that
\begin{equation}
  \label{eqn:l1gel2gelr}
  |\l_1| \ge |\l_2| \ge \cdots \ge |\l_\s| > 1
  \ge |\l_{\s+1}| \ge \cdots \ge |\l_\rho|.
\end{equation}
Theorem~\ref{theorem:hhatk}(d) and
Remark~\ref{eqn:hhatkfuncteqngeneral} tell us that for each $1\le
i\le\s$, there is a canonical height function $\hhat_{E_i}$ having
various useful properties, including
\begin{equation}
  \label{eqn:hhatEihEiO1}
  \hhat_{E_i} = h_{E_i} + O(1)
\end{equation}
and
\begin{equation}
  \label{eqn:hhatEifnsumell}
  \hhat_{E_i}\bigl(f^n(x)\bigr) = \sum_{j=0}^{\ell(i)} 
       \binom{n}{j}\l_i^{n-j}\hhat_{E_{i-j}}(x),
\end{equation}
where~$\ell(i)$ is chosen so that $E_i,E_{i-1},\ldots,E_{i-\ell(i)}$ is the
appropriate piece of the Jordan block that contains~$E_i$.
\par
On the other hand, for $\s<i\le\rho$, taking the $n^{\text{th}}$-root
of Theorem~\ref{theorem:hhatk}(a) and using the fact
that~$|\l_i|\le1$ for these~$i$, we find that
\begin{equation}
  \label{eqn:hEifnllnrho}
  \limsup_{n\to\infty} \bigl|h_{E_i}\bigl(f^n(x)\bigr)\bigr|^{1/n} \le 1
  \quad\text{for all $\s<i\le\rho$.}
\end{equation}
\par
Now take a point~$x\in X(\Kbar)$.
We first consider the case that $\hhat_{E_i}(x)\ne0$ for some $1\le
i\le\s$. We let~$k$ be the smallest such index, so
\begin{equation}
  \label{eqn:hhatEkxne0}
  \hhat_{E_k}(x)\ne0
  \quad\text{and}\quad
  \hhat_{E_{k-1}}(x) = \hhat_{E_{k-2}}(x) = \cdots = \hhat_{E_{1}}(x) = 0.
\end{equation}
Then
\begin{align}
  \label{eqn:hhatEkeqlkn}
  \hhat_{E_k}\bigl(f^n(x)\bigr)
   &= \sum_{j=0}^{\ell(k)} \binom{n}{j} \l_k^{n-j} \hhat_{E_{k-j}}(x) 
     &&\text{from \eqref{eqn:hhatEifnsumell},} \notag\\*
   &= \l_k^n \hhat_{E_k}(x)
     &&\text{from \eqref{eqn:hhatEkxne0}.}
\end{align}
This allows us to estimate 
\begin{align}
  \label{eqn:alowergeabslk}
  \alower_f(x)
  &\ge \liminf_{n\to\infty} \bigl|h_{E_k}\bigl(f^n(x)\bigr)\bigr|^{1/n}
    &&\text{from Lemma \ref{lemma:alowergehDfn},} \notag\\*
  &\ge \liminf_{n\to\infty}
   \Bigl( \bigl| \hhat_{E_k}\bigl(f^n(x)\bigr) \bigr| - O(1) \Bigr)^{1/n}
    &&\text{from \eqref{eqn:hhatEihEiO1},} \notag\\
  &= \liminf_{n\to\infty}
   \Bigl( \bigl| \l_k^n\hhat_{E_k}(x)) \bigr| - O(1) \Bigr)^{1/n}
    &&\text{from \eqref{eqn:hhatEkeqlkn},} \notag\\* 
  &= |\l_k|
    \quad\text{since $|\l_k|>1$ and $\hhat_{E_k}(x)\ne0$.}\hidewidth
\end{align}
\par
In order to find a complementary upper bound, we
recall that we fixed an ample divisor~$H\in\Div(X)$ and used it to
define~$V$. So we can write the divisor class of~$H$ in terms of our
$\CC$-basis for~$V$,
\[
  H \sim c_1E_1+\cdots+c_\rho E_\rho
  \quad\text{with $c_1,\ldots,c_\rho\in\CC$.}
\]
Since $|\l_k|>1$, we can
fix an~$\e$ satisfying $0<\e<|\l_k|-1$ and compute as follows,
where the big-$O$ constants are independent of~$n$:
\begin{align}
  h_H\bigl(f^n(x)\bigr)
  &= \sum_{i=1}^\rho c_i h_{E_i}\bigl(f^n(x)\bigr) +O(1) \notag \\
  &= \sum_{i=1}^\s c_i \hhat_{E_i}\bigl(f^n(x)\bigr) 
        + O(1) 
   + \sum_{\hidewidth i=\s+1\hidewidth}^\r c_i h_{E_i}\bigl(f^n(x)\bigr)
     \quad\text{from \eqref{eqn:hhatEihEiO1},} \notag \\
  &=  \sum_{i=1}^\s c_i \hhat_{E_i}\bigl(f^n(x)\bigr) 
        + O\bigl((1+\e)^{n}\bigr) 
     \quad\text{from \eqref{eqn:hEifnllnrho},}
  \label{eqn:hHupbdmid} \\
  &= \sum_{i=k}^\s c_i \hhat_{E_i}\bigl(f^n(x)\bigr) 
        + O\bigl((1+\e)^n\bigr) 
     \quad\text{from \eqref{eqn:hhatEkxne0},} \notag \\
  &\le O\left(\max_{k\le i\le\s} n^\rho|\l_i^n|\right)  + 
       O\bigl((1+\e)^n\bigr) 
      \quad\text{from \eqref{eqn:hhatEifnsumell},} \notag \\
  &\le O\bigl(n^\rho |\l_k|^n\bigr) + O\bigl((1+\e)^n\bigr) 
      \quad\text{from \eqref{eqn:l1gel2gelr},} \notag \\
  &\le O\bigl(n^\rho |\l_k|^n\bigr)
       \quad\text{since $\e<|\l_k|-1$.} \notag
\end{align}
Hence
\begin{equation}
  \label{eqn:aupperlelk}
  \aupper_f(x) 
  = \limsup_{n\to\infty} h_H\bigl(f^n(x)\bigr)^{1/n}  
  \le \limsup_{n\to\infty} O\bigl(n^\rho |\l_k|^n\bigr)^{1/n} = |\l_k|.
\end{equation}
Combining~\eqref{eqn:alowergeabslk} and~\eqref{eqn:aupperlelk} gives
\[
  |\l_k| \le \alower_f(x) \le \aupper_f(x) \le |\l_k|,
\]
which completes the proof that if $\hhat_{E_k}(x)\ne0$, then 
\[
  \a_f(x) = \lim_{n\to\infty} h_H\bigl(f^n(x)\bigr)^{1/n} = |\l_k|.
\]
\par
It remains to deal with the case that
\[
  \hhat_{E_1}(x) = \cdots = \hhat_{E_\s}(x) = 0,
\]
which using~\eqref{eqn:hhatEifnsumell} implies that
\begin{equation}
  \label{eqn:hhatE1Es0}
  \hhat_{E_1}\bigl(f^n(x)\bigr) = \cdots = \hhat_{E_\s}\bigl(f^n(x)\bigr) = 0
  \quad\text{for all $n\ge0$.}
\end{equation}
Substituting~\eqref{eqn:hhatE1Es0} into the earlier calculation of
$h_H\bigl(f^n(x)\bigr)$, more specifically into the line
labeled~\eqref{eqn:hHupbdmid}, we find that 
\[
  h_H\bigl(f^n(x)\bigr) = O\bigl((1+\e)^n\bigr).
\]
So taking $n^{\text{th}}$-roots and letting $n\to\infty$ gives
\[
  \aupper_f(x) = \limsup_{n\to\infty} h_H\bigl(f^n(x)\bigr)^{1/n} \le 1+\e,
\]
and since this holds for all~$\e>0$, we find that
\[
  \aupper_f(x) \le 1.
\]
But it is clear from the definition that $\alower_f(x)\ge1$, so we
conclude in this case that the limit defining $\a_f(x)$ exists and is
equal to~$1$.
\par
This completes the proof that the limit defining the arithemtic degree
exists. Further, we have shown that either~$\a_f(x)=1$, or
else~$\a_f(x)$ is equal to the absolute value of one of the
eigenvalues of~$f^*$ acting on the finite
dimensional vector space~$V\subset\Pic(X)_\QQ$.
But we also know that~$P_f(f^*)$ annihilates~$\Pic(X)_\QQ$, so the
minimal polynomial of~$f^*|_V$ must divide~$P_f(t)$, and in any case,
the eigenvalues of~$f^*|_V$ are roots of~$P_f(t)$. Hence
\[
  \bigl\{\a_f(x) : x\in X(\Kbar)\bigr\}
  \subset \{1\} 
    \cup \bigl\{|\l| : \text{$\l\in\CC$ is a root of $P_f(t)$}\bigr\},
\]
which simultaneously shows that~$\a_f(x)$ is an algebraic integer and
that~$\a_f(x)$ takes on only finitely many distinct values as~$x$ ranges
over~$X(\Kbar)$.
\par
This completes the proof of Theorem~\ref{theorem:thmA}, but we also
remark that with somewhat more work, one can show that one need only
consider eigenvalues of~$f^*:\NS(X)_\QQ\to\NS(X)_\QQ$; see
Remark~\ref{remark:afxfromNS} for further details.
\end{proof}

\section{Canonical heights for algebraic equivalence polarizations}
\label{section:algequivhts}
In this section we state an analogue of Theorem~\ref{theorem:thmB} in
which the linear equivalences are replaced by algebraic equivalences.
We omit the proof which, \emph{mutatis mutandis}, follows the proof of
Theorem~\ref{theorem:thmB}.  

\begin{theorem}
\label{theorem:thmC}
Let $X/K$ be a normal projective variety, let $f:X\to X$ be a
$K$-morphism, let $\l\in\CC$, and let
\[
  X^{(\l)}(\Kbar) = \bigl\{ x\in X(\Kbar) : \a_f(x)<|\l|^2 \bigr\}.
\]
Let
$D_0,D_1,\ldots,D_\ell\in\Div(X)_\CC$ 
be divisors that form a Jordan
block with eigenvalue~$\l$ for the linear
transformation~$f^*:\NS(X)_\CC\to\NS(X)_\CC$, i.e.,
\begin{equation}
  \label{eqn:jordanblocksintro}
  f^*D_0\equiv \l D_0,\quad
  f^*D_1\equiv D_0 + \l D_1,\quad\ldots\quad
  f^*D_\ell\equiv D_{\ell-1} + \l D_\ell,
\end{equation}
where~$\equiv$ denotes \emph{algebraic} equivalence.  Then there is a
unique function $\bfhhat_\bfD:X^{(\l)}(\Kbar)\to\CC^{\ell+1}$
satisfying the functional equation
\[
  \bfhhat_\bfD\circ f = \L \bfhhat_\bfD
\]
and the normalization condition
\begin{equation}
  \label{eqn:weakalgequivnormcond}
  \limsup_{n\to\infty} 
  \Bigl\| \bfhhat_{\bfD}\bigl(f^n(x)\bigr)
      - \bfh_{\bfD}\bigl(f^n(x)\bigr)\Bigr\|^{1/n}
  \le \a_f(x)^{1/2}.
\end{equation}
The coordinate functions of $\bfhhat_\bfD$ satisfy the
recursion relations~\eqref{eqn:hhatrecursdef} stated in
Theorem~$\ref{theorem:thmB}$.
\end{theorem}

\begin{remark}
If a divisor $D$ is algebraically equivalent to~$0$, i.e., $D\equiv0$,
then a classical height
estimate~\cite[Chapter~4,~Corollary~3.4]{lang:diophantinegeometry}
says that $h_D=o(\hplus_X)$.  This estimate is not strong enough to
prove Theorem~\ref{theorem:thmC}.  Instead one uses the stronger
estimate $h_D=O\bigl((\hplus_X)^{1/2}\bigr)$, which follows from the
N\'eron--Tate theory of canonical heights on abelian varieties;
see~\cite[Theorem~B.5.9]{hindrysilverman:diophantinegeometry}.
\end{remark}

\begin{remark}
\label{remark:hhatiflgtdf12}
Continuing with algebraic equivalence relations as in
Theorem~\ref{theorem:thmC}, if we further assume that~$f$ is dominant
and that the eigenvalue is sufficiently large, then we can obtain a
stronger result. More precisely, if the eigenvalue~$\l$ satisfies
$|\l|>\d_f^{1/2}>1$, where~$\d_f$ is the dynamical degree of~$f$,
then~\cite{kawsilv:arithdegledyndeg} tells us that
$X^{(\l)}(\Kbar)=X(\Kbar)$, and the weak normalization
condition~\eqref{eqn:weakalgequivnormcond} in
Theorem~\ref{theorem:thmC} may be replaced by the stronger condition
\[
  \bfhhat_\bfD = \bfh_\bfD + O\left( (\hplus_X)^{1/2} \right).
\]
The proof, which we omit, again follows the lines of the proof of
Theorem~\ref{theorem:thmC}, but uses a key height inequality proven
in~\cite{kawsilv:arithdegledyndeg}. (The reason that we require~$f$ to
be dominant here is because this is assumed
in~\cite{kawsilv:arithdegledyndeg}.)
\end{remark}

\begin{remark}
\label{remark:afxfromNS}
Our proof of Theorem~\ref{theorem:thmA} shows that either~$\a_f(x)=1$,
or else $\a_f(x)$ is equal to the absolute value
of an eigenvalue of the linear transformation
\[
  f^*:\Pic(X)_\QQ\to\Pic(X)_\QQ.
\]
Lemma~\ref{lemma:PffstarPicX} implies that this set of eigenvalues is
finite.  However, using Theorem~\ref{theorem:thmC} and suitably
modifying the proof of Theorem~\ref{theorem:thmA}, one can show that
in fact~$\a_f(x)$ is either~$1$ or the absolute value of an eigenvalue
of $f^*:\NS(X)_\QQ\to\NS(X)_\QQ$, so there is in fact no need to
consider the eigenvalues coming from the action of~$f^*$
on~$\Pic^0(X)_\QQ$.
\end{remark}

\section{Nef heights and arithmetic degrees for endomorphisms 
of abelian varieties}
\label{section:arithdegabelianvar}

Our primary goal in this section is to prove that for an abelian
variety~$A$, an isogeny~$f:A\to A$, and a point~$x\in A$, all defined
over~$\Qbar$, if the $f$-orbit~$\Orbit_f(x)$ of~$x$ is Zariski dense
in~$A$, then $\a_f(x)=\d_f$. Along the way we describe the set of
points satisfying $\qhat_{A,D}(x)=0$ for a nonzero nef divisor~$D$,
generalizing the classical result for ample divisors. We will make
extensive use of the geometry of abelian varieties and their
N\'eron--Severi groups and endomorphism algebras as described
in~\cite[Sections~19--21]{MR0282985}.

For this section we set the following notation:

\begin{notation}
\item[$A/\Qbar$]
an abelian variety defined over $\Qbar$
\item[$\hat A$]
the dual of the abelian variety~$A$, i.e., $\hat A=\Pic^0(A)$.
\item[$\d_f$]
the dynamical degree of an endomorphism $f\in\End(A)$, which by
definition is the spectral radius of the induced map $f^*$ on
$\NS(A)_\CC$.
\item[$H$]
an ample divisor on $A$.
\item[$H_{\a,\b}$]
For $\a,\b\in\End(A)_\QQ$, the divisor
\[
  \hspace*{6em}
  H_{\a,\b} = (\a\pi_1+\b\pi_2)^*H - (\a\pi_1)^*H - (\b\pi_2)^*H
  \in \Div(A^2)_\QQ,
\]
where $\pi_1,\pi_2:A\times A\to A$ are the projection maps.
\item[$\f_D$]
For a divisor class $[D]\in\NS(A)_\QQ$, the map
\[
  \hspace*{4em}
  \f_D:A\longrightarrow\hat A=\Pic^0(A),\qquad
  \f_D(x) = [T_x^*D - D],
\]
where $T_x:A\to A$ is the translation-by-$x$ map;
see~\cite[page~60]{MR0282985}.
\item[$\F$]
The inclusion 
\begin{equation}
  \label{eqn:FHSAtoEndA}
  \hspace*{4em}
  \F : \NS(A)_\QQ \longhookrightarrow \End(A)_\QQ,\qquad
  \F_D = \f_H^{-1}\circ\f_D,
\end{equation}
induced by the ample divisor~$H$;
see~\cite[pages 190, 208]{MR0282985}.
\item[$\qhat_{A,D}$]
the quadratic part
of the canonical height on $A(\Qbar)$ relative to the divisor~$D$,
defined by $\qhat_{A,D}(x)=\lim n^{-2}h_{A,D}(nx)$.
\item[$\langle\,\cdot\,,\,\cdot\,\rangle_{A,D}$]
the associated height pairing on~$A(\Qbar)$, defined by
\[
  \hspace*{4em}
  \langle x,y\rangle_{A,D}=\qhat_{A,D}( P+Q)-\qhat_{A,D}( P)-\qhat_{A,D}( Q).
\]
We extend the pairing $\RR$-linearly to~$A(\Kbar)_\RR$.
\item[$\hat\a$]
the induced map $\hat\a:\hat B\to\hat A$ for
a homomorphism~$\a:A\to B$ of abelian varieties.
\item[$\a'$]
the Rosati involution on~$A$ associated to~$H$, defined by
\begin{equation}
  \label{eqn:rosatidef}
  \hspace*{4em}
  \a' = \f_H^{-1}\circ \hat\a\circ \f_H,
  \quad\text{where $\a\in\End(A)$.}
\end{equation}
\end{notation}

We begin with a number of geometric results that will be used
as input to the height machinery.

\begin{lemma}
\label{lemma:ffDhatffDfx}
Let $A$ and $B$ be abelian varieties, let
$f:B\to A$ be an isogeny, and let $D\in\NS(B)_\QQ$. Then
\begin{equation}
  \label{eqn:ffstardx}
  \f_{f^*D} = \hat f \circ \f_D \circ f.
\end{equation}
If further~$A=B$ and we fix an ample divisor~$H$ to define the
inclusion~\eqref{eqn:FHSAtoEndA} and the Rosati
involution~\eqref{eqn:rosatidef}, then
\begin{equation}
  \label{eqn:FastarDaFDa}
  \F_{\a^*D} = \a' \circ \F_D \circ \a
  \quad\text{for all $\a\in\End(A)_\QQ$.}
\end{equation}
\end{lemma}
\begin{proof}
For $x\in B$ we compute
\begin{align*}
  \f_{f^*D}(x)
  &= \bigl[T_x^*(f^*D)-f^*D\bigr] \\
  &= \bigl[f^*(T_{f(x)}^*D)-f^*D\bigr] \\
  &= f^*\bigl[T_{f(x)}^*D-D\bigr] \\
  &= \hat f\bigl(\f_D\bigl(f(x)\bigr)\bigr).
\end{align*}
Hence $\f_{f^*D} = \hat f \circ \f_D \circ f$, which
proves~\eqref{eqn:ffstardx}. 
\par
In the case that~$A=B$, we have
\begin{align*}
  \F_{\a^*D} 
  &= \f_H^{-1}\circ\f_{\a^*D}
    \quad\text{definition of $\F$, see~\eqref{eqn:FHSAtoEndA},} \\
  &= \f_H^{-1}\circ\hat\a\circ\f_{D}\circ\a
    \quad\text{from \eqref{eqn:ffstardx},} \\
  &= \f_H^{-1}\circ\hat\a\circ\f_H\circ\f_H^{-1}\circ\f_{D}\circ\a \\
  &= \a'\circ\F_D\circ\a
    \quad\text{definitions~\eqref{eqn:FHSAtoEndA}
     and~\eqref{eqn:rosatidef} of $\F$ and Rosati,}
\end{align*}
which completes the proof of~\eqref{eqn:FastarDaFDa}.
\end{proof}

\begin{lemma}
\label{lemma:HabHbarbbara}
Let $\a\in\End(A)_\RR$. Then
\[
  H_{\a,1}\equiv H_{1,\a'} \quad\text{in $\NS(A^2)_\RR$.}
\]
\end{lemma}
\begin{proof}
We remark that\footnote{This is fairly standard.
  Briefly,~\cite[Proposition~A.7.3.2]{hindrysilverman:diophantinegeometry}
  says that $(\pi_1+\pi_2)^*=\pi_1^*+\pi_2^*$ on~$\Pic^0(A)$, so
  $(\a+\b)^*=\bigl((\pi_1+\pi_2)\circ(\a\times\b)\bigr)^* =
  (\a\times\b)^*\circ(\pi_1+\pi_2)^* =
  (\a\times\b)^*\circ(\pi_1^*+\pi_2^*) =
  \bigl(\pi_1\circ(\a\times\b)\bigr)^*+\bigl(\pi_2\circ(\a\times\b)\bigr)^*
  = \a^*+\b^*$ on~$\Pic^0(A)$.}
\begin{equation}
  \label{eqn:ahataishom}
  \text{$\a\to\hat\a$\quad is a ring homomorphism\quad
     $\End(A)\to\End(\hat A)$.}
\end{equation}
We let~$\a,\b\in\End(A)_\RR$ and compute more generally
\begin{align*}
  \f_{H_{\a,\b}}
  &= \f_{(\a\pi_1+\b\pi_2)^*H} - \f_{(\a\pi_1)^*H} - \f_{(\b\pi_2)^*H}
    \quad\text{definition of $H_{\a,\b}$,} \\
  &= (\widehat{\a\pi_1+\b\pi_2})\circ\f\circ(\a\pi_1+\b\pi_2)
     - (\widehat{\a\pi_1})\circ\f\circ(\a\pi_1) \\*
  &\omit\hfill\qquad${} - (\widehat{\b\pi_2})\circ\f\circ(\b\pi_2)$
  \quad\text{from Lemma~\ref{lemma:ffDhatffDfx},  
                  equation~\eqref{eqn:ffstardx},} \\
  &= (\widehat{\a\pi_1}+\widehat{\b\pi_2})\circ\f\circ(\a\pi_1+\b\pi_2)
     - (\widehat{\a\pi_1})\circ\f\circ(\a\pi_1) \\*
  &\omit\hfill\qquad${} - (\widehat{\b\pi_2})\circ\f\circ(\b\pi_2)$
    \quad\text{from \eqref{eqn:ahataishom},} \\
  &= \widehat{\a\circ\pi_1}\circ\f_H\circ\b\circ\pi_2
   + \widehat{\b\circ\pi_2}\circ\f_H\circ\a\circ\pi_1 \\
  &= \hat\pi_1\circ\hat\a\circ\f_H\circ\b\circ\pi_2
   + \hat\pi_2\circ\hat\b\circ\f_H\circ\a\circ\pi_1 \\
  &= \hat\pi_1\circ\f_H\circ\a'\circ\b\circ\pi_2
   + \hat\pi_2\circ\f_H\circ\b'\circ\a\circ\pi_1 \\*
  &\omit\hfill\text{definition of the Rosati involution.}
\end{align*}
Hence
\[
  \f_{H_{\a,1}}
  = \hat\pi_1\circ\f_H\circ\a'\circ\pi_2
   + \hat\pi_2\circ\f_H\circ\a\circ\pi_1
\]
and
\[
  \f_{H_{1,\a'}}
  = \hat\pi_1\circ\f_H\circ\a'\circ\pi_2
   + \hat\pi_2\circ\f_H\circ\a''\circ\pi_1.
\]
Since $\a''=\a$, this shows that $\f_{H_{\a,1}}=\f_{H_{1,\a'}}$. To
complete the proof that $H_{\a,1}\equiv H_{1,\a'}$, it suffices to
note that the map
\[
  \NS(X)_\RR\to\End(X)_\RR,
  \quad D\mapsto\f_D,
\]
is injective~\cite[page~208]{MR0282985}. 
\end{proof}

\begin{proposition}
\label{proposition:DnefiffFda2}
Let $D\in\NS(A)_\RR$ be a nef divisor.
Then there is an endomorphism $\a\in\End(A)_\RR$ satisfying
\[
  \F_D = \a'\circ\a \quad\text{and}\quad \a'=\a.
\]
\end{proposition}
\begin{proof}
The~$\RR$-algebra~$\End(A)_\RR$ is isomorphic to a product of matrix
algebras of the form~$\Mcal_n(\RR)$,~$\Mcal_n(\CC)$, and
$\Mcal_n(\HH)$, and the isomorphism may be chosen so that the Rosati
involution on~$\End(A)_\RR$ corresponds to the standard involution
$T\to {}^t\bar T$ on the matrix algebras;
cf.\ \cite[pages~208--209]{MR0282985}. (Here~$t\to\bar t$ is the
identity on~$\RR$, complex conjugation on~$\CC$, and quaternionic
conjugation on~$\HH$.) The map~$\F$ gives an
isomorphism~\cite[page~208]{MR0282985},
\[
  \F : \NS(A)_\RR
  \xrightarrow{\;\sim\;} \bigl\{\a\in\End(A)_\RR : \a'=\a \bigr\},
\]
so~$\NS(A)_\RR$ is identified with a product of Jordan algebras of the
form $\Hcal_n(\RR)$, $\Hcal_n(\CC)$, and~$\Hcal_n(\HH)$, where
\[
  \Hcal_n(\KK)=\{T\in\Mcal_n(\KK):{}^t\bar T=T\}
\]
denotes the set of Hermitian matrices for~$\KK=\RR$,~$\CC$,
or~$\HH$~\cite[Theorem~6, page~208]{MR0282985}.  
\par
The matrices in~$\Hcal_n(\KK)$ have real eigenvalues, since they are
self-adjoint. It is proven in~\cite[page~210]{MR0282985} that a
divisor~$D$ is ample if and only if the eigenvalues associated
to~$\F_D$ are all strictly positive. Since the nef cone is the closure
of the ample cone in~$\NS(A)_\RR$, we see that~$D$ is nef if and only
if all of the eigenvalues associated to~$\F_D$ are non-negative.
Equivalently, $D$ is nef if and only if
the matrices associated to~$\F_D$ are self-adjoint
and positive semi-definite.
\par
A standard result in linear algebra says that a self-adjoint positive
semi-definite matrix~$T\in\Hcal_n(\KK)$ can be written in the
form~$T={}^t\bar SS$ for some~$S\in\Hcal_n(\KK)$.  See, e.g.,
\cite[Theorem~7.27]{MR1482226} for the cases~$\KK=\RR$ and~$\KK=\CC$,
and~\cite[Corollary~2.6]{MR2990115}
for~$\KK=\HH$.\footnote{\cite[Corollary~2.6]{MR2990115} actually says
  that~$T$ is unitarily equivalent to a diagonal matrix
  in~$\Mcal_n(\RR)$.  So~$T={}^t\bar U\D U$ with~${}^t\bar U=U^{-1}$
  and~$\D$ diagonal and real. The positive semi-definiteness of~$T$
  implies that~$\D$ has non-negative entries, so~$\D$ has a square
  root in~$\Mcal_n(\RR)$, say~$\D=\G^2$. Since~${}^t\bar\G=\G$, it
  follows that~$T={}^t\bar SS$ with~$S={}^t\bar U\G U$
  satisfying~${}^t\bar S=S$.}  Hence with the indicated
identifications, we can find an~$\a\in\End(A)_\RR$
satisfying~$\F_D=\a'\circ\a$ and $\a'=\a$.
\end{proof}

We now turn to some arithmetic consequences of these geometric
facts.

\begin{proposition}
\label{proposition:axbybxay}
Let $\a\in\End(A)_\RR$ and let $x,y\in A(\Qbar)$. Then
\[
  \bigl\langle \a(x),y\big\rangle_{A,H}
  =   \bigl\langle x,\a'(y)\big\rangle_{A,H}.
\]
\end{proposition}
\begin{proof}
We first compute
\begin{align}
  \label{hA2Habxypaired}
  \qhat_{A^2,H_{\a,\b}}(x,y)
  &= \qhat_{A^2,(\a\pi_1+\b\pi_2)^*H}(x,y)
     - \qhat_{A^2,(\a\pi_1)^*H}(x,y)
       \notag\\*
  &\omit\hfill${}- \qhat_{A^2,(\b\pi_2)^*H}(x,y)$
    \quad by definition of $H_{\a,\b}$, \notag\\
  &= \qhat_{A,H}\bigl(\a(x)+\b(y)\bigr)
     - \qhat_{A,H}\bigl(\a(x)\bigr) - \qhat_{A,H}\bigl(\b(y)\bigr) \notag\\
  &= \bigl\langle \a(x),\b(y) \bigr\rangle_{A,H}.
\end{align}
Hence
\[
  \bigl\langle \a(x),y \bigr\rangle_{A,H} = \qhat_{A^2,H_{\a,1}}(x,y)
  \quad\text{and}\quad
  \bigl\langle x,\a'(y) \bigr\rangle_{A,H} = \qhat_{A^2,H_{1,\a'}}(x,y).
\]
But Lemma~\ref{lemma:HabHbarbbara} says that $H_{\a,1}\equiv
H_{1,\a'}$, and the (quadratic part of the) canonical height on an
abelian variety depends on only the algebraic equivalence class of the
divisor, which completes the proof of
Proposition~\ref{proposition:axbybxay}.
\end{proof}

\begin{proposition}
\label{proposition:xyADxFDyAH}
Let $D\in\Div(A)_\RR$ and $x,y\in A(\Qbar)_\RR$. Then
\[
  \langle x,y\rangle_{A,D} = \bigl\langle x,\F_D(y)\big\rangle_{A,H}.
\]
\end{proposition}
\begin{proof}
For~$E\in\Div(A)$ and~$x,z\in A$, we compute
\begin{align}
  \label{eqn:hAfEzx}
  \qhat_{A,\f_E(z)}(x)
  &= \qhat_{A,T_z^*E-E}(x) 
    \quad\text{definition of $\f_E$,} \notag\\
  &= \qhat_{A,T_z^*E}(x) - \qhat_{A,E}(x)
     \quad\text{linearity,}  \notag\\
  &= \qhat_{A,E}\bigl(T_z(x)\bigr) - \qhat_{A,E}\bigl(T_z(0)\bigr)
         - \qhat_{A,E}(x) \hspace{4em}\notag\\
  &\omit\hfil
   \text{from \cite[Theorem~B.5.6(d)]{hindrysilverman:diophantinegeometry},}
       \notag\\
  &=\qhat_{A,E}(x+z) - \qhat_{A,E}(z) - \qhat_{A,E}(x) \notag\\
  &= \langle x,z\rangle_{A,E}.
\end{align}
Applying \eqref{eqn:hAfEzx} twice, we find that
\begin{align*}
  \bigl\langle x,\F_D(y)\big\rangle_{A,H}
  &= \qhat_{A,\f_H\circ\F_D(y)}(x) \\
  &\omit\hfill\text{from \eqref{eqn:hAfEzx} with $z=\F_D(y)$ and $E=H$,} \\
  &= \qhat_{A,\f_D(y)}(x)
     \quad\text{since $\F_D=\f_H^{-1}\circ\f_D$,} \\
  &= \langle x,y\rangle_{A,D}.
     \quad\text{from \eqref{eqn:hAfEzx} with $z=y$ and $E=D$.} 
\end{align*}

This completes the proof of Proposition~\ref{proposition:xyADxFDyAH}.
\end{proof}

We now have the tools needed to prove Theorem~\ref{theorem:hADx0ax0},
which we restate as the first part of the following theorem.

\begin{theorem} 
\label{theorem:hADx0ax02}
Let $A/\Qbar$ be an abelian variety defined over~$\Qbar$, let
$D\in\Div(A)_\RR$ be a nonzero nef divisor, and
let~$\qhat_{A,D}$ be the quadratic part of the
canonical height on~$A$ with respect to~$D$.  
\begin{parts}
\Part{(a)}
There is a unique abelian subvariety $B_D\subsetneq A$ such that
\[
  \bigl\{ x\in A(\Qbar) : \qhat_{A,D}(x)=0 \bigr\}
  = B_D(\Qbar) + A(\Qbar)_\tors.
\]
\Part{(b)}
Let $f\in\End(A)$ and suppose that $f^*D\equiv\l D$ in $\NS(A)_\RR$.
Then the abelian subvariety~$B_D$ from~\textup{(a)} is~$f$-invariant, i.e., 
$f(B_D)\subset B_D$.
\Part{(c)}
Let $K/\QQ$ be a number field over which~$A$ and~$D$ are defined. Then
\[
  \bigl\{ x\in A(K) : \qhat_{A,D}(x)=0 \bigr\}
\]
is not Zariski dense in~$A$.
\end{parts}
\end{theorem}

\begin{proof}[Proof of Theorem~$\ref{theorem:hADx0ax0}$]
Since~$D$ is nef, we can use Proposition~\ref{proposition:DnefiffFda2}
to find an $\a\in\End(A)_\RR$ (depending on~$D$) satisfying
\begin{equation}
  \label{eqn:aprimeaFDa2}
  \F_D=\a'\circ\a \quad\text{and}\quad \a'=\a.
\end{equation}
We compute
\begin{align*}
  \qhat_{A,D}(x)
  &= \frac12 \langle x,x\rangle_{A,D} \\*
  &= \frac12 \bigl\langle x,\F_D(x)\big\rangle_{A,H}
     \quad\text{from Proposition~\ref{proposition:xyADxFDyAH}(b),} \\
  &= \frac12  \bigl\langle x,\a'\circ\a(x)\big\rangle_{A,H}
     \quad\text{from \eqref{eqn:aprimeaFDa2},} \\
  &= \smash[b]{
         \frac12  \bigl\langle \a(x),\a(x)\big\rangle_{A,H}
     }
     \quad\text{from Proposition \ref{proposition:axbybxay} with $y=\a(x)$}\\*
  &=  \qhat_{A,H}\bigl(\a(x)\bigr).
\end{align*}
But~$\qhat_{A,H}$ is the (quadratic part of the) canonical height
on~$A$ relative to an ample divisor, so it is a positive definite
quadratic form on $A(\Qbar)\otimes\RR$;
see~\cite[Proposition~B.5.3]{hindrysilverman:diophantinegeometry}.
(Recall that the canonical height pairing on~$A(\Qbar)$
is extended $\RR$-linearly to~$A(\Qbar)_\RR$.)
Hence
\begin{equation}
  \label{eqn:hhatADx0aix0}
  \qhat_{A,D}(x)=0
  \quad\Longleftrightarrow\quad
  \a(x)=0\quad\text{in $A(\Qbar)_\RR$.}
\end{equation}

To complete the proof of Theorem~\ref{theorem:hADx0ax0}, we use
the following elementary linear algebra result.

\begin{lemma}
\label{lemma:ax0iffbix0}
Let $V$ be a~$\QQ$-vector space, and fix a $\QQ$-vector subspace
$\Dcal\subset\End(V)$ of the ring of $\QQ$-linear
endomorphisms of~$V$.  Let $\a\in\Dcal\otimes\RR$ be an $\RR$-linear
transformation of~$V\otimes\RR$. Then there is a finite collection of
endomorphisms $\b_1,\ldots,\b_r\in\Dcal$ with the property that
for~$\bfv\in V$,
\[
  \text{$\a(\bfv)=\bfzero$ in $V\otimes\RR$}
  \quad\Longleftrightarrow\quad
  \text{$\b_1(\bfv)=\cdots=\b_r(\bfv)=\bfzero$ in $V$.}
\]
\end{lemma}
\begin{proof}
We write $\a\in\Dcal\otimes\RR$ as a sum
\[
  \a = \sum_{i=1}^r c_i\b_i
  \quad\text{with $c_i\in\RR$ and $\b_i\in\Dcal$.}
\]
We may assume that~$r$ is chosen to be minimal, which implies
that\footnote{If not, then after relabeling, we can write
  $c_1=\sum_{i=2}^r b_ic_i$ with $b_i\in\QQ$, so $ \a = \sum_{i=2}^r
  c_i(b_i\b_1+\b_i)$ with $b_i\b_1+\b_i\in\Dcal$, contradicting the
  minimality of~$r$.}
\[
  \text{$c_1,\ldots,c_r\in\RR$ are~$\QQ$-linearly
       independent.}
\]
This in turn implies that for $\bfv_1,\ldots,\bfv_r\in V$ we
have\footnote{Let $\{\bfe_j\}_{j\in\Jcal}$ be a $\QQ$-basis for~$V$,
  so $\{\bfe_j\}_{j\in\Jcal}$ is also automatically an~$\RR$-basis
  for~$V\otimes\RR$. For each~$1\le i\le r$, write
  $\bfv_i=\sum_{j\in\Jcal}b_{ij}\bfe_j$ with all $b_{ij}\in\QQ$ and
  almost all~$b_{ij}=0$.  Then $\bfzero=\sum_i c_i\bfv_i = \sum_{j\in\Jcal}
  \left(\sum_i c_i b_{ij}\right) \bfe_j$, so the assumptions that
  the~$\bfe_j$ are~$\RR$-linearly independent and the~$c_i$
  are~$\QQ$-linearly independent implies that~$b_{ij}=0$ for
  all~$i,j$.}
\begin{equation}
  \label{eqn:sumcivi0iffvi0}
  \sum_{i=1}^r c_i\bfv_i=\bfzero\quad\text{in $V\otimes\RR$}
  \quad\Longleftrightarrow\quad
  \bfv_1=\cdots=\bfv_r=\bfzero.
\end{equation}
Then for~$\bfv\in V$ we have
\begin{align*}
  \text{$\a(\bfv)=\bfzero$ in $V\otimes\RR$}
  &\quad\Longleftrightarrow\quad
  \text{$\displaystyle\sum_{i=1}^r c_i\b_i(\bfv)=\bfzero$ in $V\otimes\RR$} \\*
  &\quad\Longleftrightarrow\quad
  \text{$\b_1(\bfv)=\cdots=\b_r(\bfv)=\bfzero$ in $V$,}
\end{align*}
where for the last implication we have used~\eqref{eqn:sumcivi0iffvi0}
and the fact that $\b_1(\bfv),\ldots,\b_i(\bfv)$ are in~$V$.
\end{proof}

We now resume the proof of Theorem~\ref{theorem:hADx0ax02}.  We let
$\a\in\End(A)_\RR$ be the nonzero endomorphism appearing
in~\eqref{eqn:aprimeaFDa2}.  We apply Lemma~\ref{lemma:ax0iffbix0} to
the $\QQ$-vector space~$V=A(\Qbar)_\QQ$ and the $\QQ$-subspace
\[
  \Dcal = \End(A)_\QQ \subset \End(V).
\]
(Here~$\End(A)$ denotes the ring of algebraic maps~$A\to A$,
while~$\End(V)$ denotes the ring of~$\QQ$-linear maps~$V\to V$.)
Lemma~\ref{lemma:ax0iffbix0} says that we can find endomorphisms
$\b_1,\ldots,\b_r\in\End(A)_\QQ$ so that for $x\in A(\Qbar)_\QQ$,
\begin{equation}
  \label{eqn:aix0bijx0}
  \a(x)=0
  \quad\Longleftrightarrow\quad
  \b_1(x)=\cdots=\b_r(x)=0.
\end{equation}
Combining~\eqref{eqn:hhatADx0aix0} and~\eqref{eqn:aix0bijx0} yields
\begin{equation}
  \label{eqn:hADx0bi0}
  \qhat_{A,D}(x)=0
  \quad\Longleftrightarrow\quad
  \b_1(x)=\cdots=\b_r(x)=0\quad\text{in $A(\Qbar)_\QQ$.}
\end{equation}
Replacing each of the finitely many~$\b_{i}\in\End(A)_\QQ$ by an
appropriate non-zero integral multiple $m_{i}\b_{i}$, we may assume
that the~$\b_{i}$ all lie in~$\End(A)$.
\par
We let
\[
  B = \bigcap_{i=1}^r \Ker(\b_{i}).
\]
We note that~$B$ is a (not necessarily connected) algebraic subgroup
of~$A$, and further, $B\ne A$, since~$\a\ne0$, so at least
one~$\b_i\ne0$. (This is where we use the assumption that~$D\ne0$,
which ensures that~$\F_D\ne0$, so~$\a\ne0$.)  
The definition of~$B$ and~\eqref{eqn:hADx0bi0} imply that
\[
  \bigl\{x\in A(\Qbar)_\QQ : \qhat_{A,D}(x)=0 \bigr\}
  = B(\Qbar)_\QQ.
\]
It follows that
\[
  \bigl\{x\in A(\Qbar) : \qhat_{A,D}(x)=0 \bigr\}  
  = B(\Qbar)^\div,
\]
where
\[
   B(\Qbar)^\div
   = \bigl\{y\in A(\Qbar) : my\in B(\Qbar)~\text{for some $m\ge1$} \bigr\}
\]
is the divisible hull of~$B$ in~$A$.  Without loss of generality, we
may replace~$B$ by its connected component, since that won't change
the group~$B(\Qbar)^\div$.  But if~$B$ is connected, then it is easy
to see that\footnote{One inclusion is clear. For the other, let $x\in
  B(\Qbar)^\div$. Then $mx\in B$ for some $m\ge1$. Since~$B$ is
  connected, there is a $y\in B(\Qbar)$ with $my=mx$. Then $x=y+(x-y)$
  with $y\in B(\Qbar)$ and $x-y\in A(\Qbar)_\tors$.}
\[
  B(\Qbar)^\div = B(\Qbar)+A(\Qbar)_\tors.
\]
In order to complete the proof of Theorem~\ref{theorem:hADx0ax02}(a),
it remains only to show that~$B$ is uniquely determined by~$D$.
\par
So suppose that $B_1\subset A$ and $B_2\subset A$ are abelian
subvarieties of~$A$ satisfying $B_1+A_\tors=B_2+A_\tors$.  We may
view~$B_2+A_\tors$ as a scheme over $\Spec(\Qbar)$, and our assumption
implies that $B_1$ is a $\Qbar$-subscheme of $B_2+A_\tors$.  But~$B_1$
is a scheme of finite type over $\Spec(\Qbar)$, so it is contained in
a subscheme of $B_2+A_\tors$ of finite type.  Hence there is an
integer $m\ge1$ such that $B_1\subset B_2+A[m]$. But~$B_1$ and~$B_2$
are irreducible, so $B_1\subset B_2+t$ for some $t\in A[m]$. Finally,
the fact that~$B_1$ and~$B_2$ contain~$0$ implies that~$B_1\subset
B_2$.  Reversing the roles of~$B_1$ and~$B_2$ gives the opposite
inclusion, so $B_1=B_2$.
\par\noindent(b)\enspace
We have
\begin{equation}
  \label{eqn:aAFflhAF}
  \qhat_{A,D}\circ f = \qhat_{A,f^*D} = \qhat_{A,\l D} = \l\qhat_{A,D},
\end{equation}
where the middle equality follows from the fact that the quadratic part
of the canonical height depends only on the algebraic equivalence class
of the divisor. Hence applying~(a) twice, we find that
\begin{align*}
  f(B_D)
  &\subset f(B_D+A_\tors)\\
  &= f\bigl( \{x\in A : \qhat_{A,D}(x)=0 \} \bigr)\\
  &\subset \{x\in A : \qhat_{A,D}(x)=0 \}\\
  &= B_D+A_\tors.
\end{align*}
But as in the proof of~(a), the image~$f(B_D)$ is a subscheme of
$B_D+A_\tors$ that is of finite type over~$\Spec(\Qbar)$, so~$f(B_D)$
is contained in $B_D+A[m]$ for some~$m\ge1$. But~$f(B_D)$ is irreducible,
so it is contained in~$B_D+t$ for some~$t\in A[m]$, and finally the
fact that~$0$ is in both~$B_D$ and~$f(B_D)$ implies that
$f(B_D)\subset B_D$.
\par\noindent(c)\enspace
Let
\[
  \G = A(K) \cap \bigl(B_D(\Qbar)+A(\Qbar)_\tors\bigr).
\]
It follows~(a) that
\[
  \bigl\{ x\in A(K) : \qhat_{A,D}(x)=0 \bigr\} \subset \G,
\]
so in order to prove~(c), it suffices to show that~$\G$ is not
Zariski dense in~$A$.
\par
The group~$A(K)$ is finitely
generated~\cite[Theorem~C.0.1]{hindrysilverman:diophantinegeometry},
so its subgroup~$\G$ is also finitely generated. Let~$y_1,\ldots,y_t$
be generators for~$\G$. Each~$y_i$ is in~$B_D+A_\tors$, so we can find
an integer~$m\ge1$ such that $my_1,\ldots,my_t\in B_D$. It follows
that~$m\G\subset B_D$, and hence~$\G\subset B_D+A[m]$ is not Zariski
dense in~$A$.
\end{proof}

\begin{lemma}
\label{lemma:existseigendiv}
Let $f:X\to X$ be a morphism of a normal projective variety.  Then
there exists a non-zero nef eigendivisor~$F\in\NS(X)_\RR$ satisfying
$f^*F\equiv \d_f F$.
\end{lemma}
\begin{proof}
The existence of~$F$ follows from an elementary
Perron--Fro\-be\-nius-type result of Birkhoff~\cite{MR0214605} applied to
the vector space $\NS(X)_\RR$, the linear transformation~$f^*$, and
the nef cone in~$\NS(X)_\RR$.
See~\cite[Remark~28]{kawsilv:arithdegledyndeg}.
\end{proof}

\begin{corollary}
\label{corollary:Ofxdenseafxdf}
Let $f\in\End(A)$, let $F\in\Div(A)_\RR$ be a divisor
satisfying~$f^*F\equiv\d_fF$ as described
in Lemma~$\ref{lemma:existseigendiv}$,
let $x\in A(\Qbar)$, and let $\overline{\Orbit_f(x)}$ denote the
Zariski closure in~$A$ of the $f$-orbit of~$x$. Then 
the following implications hold\textup:
\[
  \overline{\Orbit_f(x)}=A
  \quad\Longrightarrow\quad
  \qhat_{A,F}(x)>0
  \quad\Longrightarrow\quad
  \a_f(x)=\d_f.
\]
\end{corollary}

\begin{proof}
Applying~\eqref{eqn:aAFflhAF} with $\l=\d_f$, we have
\[
  \qhat_{A,F}\circ f = \d_f\qhat_{A,F}.
\]
Further, if $D$ is symmetric, then we have $\qhat_{A,D}=h_{A,D}+O(1)$,
so
\begin{equation}
  \label{eqn:hAF12hAFneg1F}
  \qhat_{A,F} = \frac12 \qhat_{A,F+[-1]^*F}
  = \frac12 h_{A,F+[-1]^*F} + O(1).
\end{equation}
We also note that the canonical height associated to the nef
divisor~$F$ is non-negative, because $F+\e H$ is ample for any $\e>0$
and any ample divisor~$H$, so
\[
  \qhat_{A,F} = \qhat_{A,F+\e H} - \qhat_{A,\e H} \ge - \e \qhat_{A,H}.
\]
Since $\qhat_{A,H}\ge0$ and~$\e$ is arbitrary, we see that~$\qhat_{A,F}\ge0$.
\par
We now suppose that~$\qhat_{A,F}(x)>0$ and compute
\begin{align*}
  \a_f(x)
  &=\alower_f(x) 
    \quad\text{from Theorem~\ref{theorem:thmA},} \\
  &\ge \liminf_{n\to\infty} \bigl|h_{A,F+[-1]^*F}\bigl(f^n(x)\bigr)\bigr|^{1/n}
    \quad\text{from Lemma~\ref{lemma:alowergehDfn},} \\
  &= \liminf_{n\to\infty} \qhat_{A,F}\bigl(f^n(x)\bigr)^{1/n} 
    \quad\text{from \eqref{eqn:hAF12hAFneg1F},}\\
  &= \liminf_{n\to\infty} \bigl(\d_f^n\qhat_{A,F}(x)\bigr)^{1/n} 
    \quad\text{from \eqref{eqn:aAFflhAF},}\\
  &= \d_f
     \quad\text{since $\qhat_{A,F}(x)>0$.}
\end{align*}
Since~\cite{kawsilv:arithdegledyndeg} says that we always
have~$\a_f(x)\le\d_f$,\footnote{The proof of $\a_f(x)\le\d_f$
  in~\cite{kawsilv:arithdegledyndeg} assumes that~$f$ is dominant. But
  in our situation, since~$f^n(A)$ is a sequence of abelian
  subvarieties of~$A$ of non-increasing dimension, it eventually
  stabilizes, say~$B=f^m(A)$ with $f:B\to B$ an isogeny. Then one
  easily checks that $\a_f(x)=\a_{f|_B}\bigl(f^m(x)\bigr)$ and
  $\d_f=\d_{f|_B}$, which reduces us to the case of a dominant map.}
this proves the implication
\[
  \qhat_{A,F}(x)>0
  \quad\Longrightarrow\quad
  \a_f(x)=\d_f.
\]
\par
Next we let~$K/\QQ$ be a number field such that~$A$,~$D$, and~$f$ are
defined over~$K$ and such that~$x\in A(K)$. Then~$\Orbit_f(x)\subset
A(K)$. Now suppose that $\qhat_{A,F}(x)=0$. Then~\eqref{eqn:hAF12hAFneg1F}
tells us that $\qhat_{A,F}\bigl(f^n(x)\bigr)=\d_f\qhat_{A,F}(x)=0$
for all $n\ge0$, so
\[
  \Orbit_f(x)
  \subset \bigl\{y\in A(K) : \qhat_{A,F}(y) = 0 \bigr\}.
\]
Theorem~\ref{theorem:hADx0ax02}(c) tells us that the set on the right
is not Zariski dense in~$A$, so the same is true
of~$\Orbit_f(x)$. This completes the proof of the implication
\[
  \qhat_{A,F}(x)=0
  \quad\Longrightarrow\quad
  \overline{\Orbit_f(x)} \ne A,
\]
which combined with the fact that~$\qhat_{A,F}(x)\ge0$ gives the other
desired implication
\[
  \overline{\Orbit_f(x)} = A
  \quad\Longrightarrow\quad
  \qhat_{A,F}(x)>0.
\]
This completes the proof of Corollary~\ref{corollary:Ofxdenseafxdf}.
\end{proof}

\begin{remark}
We give examples to show that neither of the implications in
Corollary~\ref{corollary:Ofxdenseafxdf} is true in the opposite
direction.
Let~$E$ be an elliptic curve, and take
\[
  A=E^2,\quad
   H = \pi_1^*(O)+\pi_2^*(O),\quad\text{and}\quad f(P,Q)=(2P,2Q).
\]
Then~$H$ is ample and symmetric (so $\qhat_{A,H}=\hhat_{A,H}$) and
satisfies $f^*H\sim4H$, so in particular~$\d_f=4$. Further,
\[
  \hhat_{A,H}(P,Q)=\hhat_{E,(O)}(P)+\hhat_{E,(O)}(Q).
\]
Then for any nontorsion point $P\in E$ we have
\[
  \hhat_{A,H}(P,O) = \hhat_{E,(O)}(P)+\hhat_{E,(O)}(O) = \hhat_{E,(O)}(P) > 0
\]
and
\[
  \overline{\Orbit_f(P,O)} = E\times\{O\},
\]
which shows that the implication
$\overline{\Orbit_f(x)}=A\;\Rightarrow\;\qhat_{A,F}(x)>0$ cannot be
reversed.
\par
Continuing with the assumption that~$P\notin E_\tors$, we compute
\begin{align*}
  h_{A,H}\bigl(f^n(P,O)\bigr) 
  &= \hhat_{A,H}\bigl(f^n(P,O)\bigr)+O(1) \\
  &= 4^n\hhat_{A,H}(P,O)+O(1) \\
  &= 4^n\hhat_{E,(O)}(P)+O(1),
\end{align*}
so the fact that $\hhat_{E,(O)}(P)>0$ and the definition
of arithmetic degree give~$\a_f(P,O)=4=\d_f$.
On the other hand, consider
the nonzero nef divisor~$F=\pi_2^*(O)$. It satisfies
\[
   f^*F\sim4F
   \quad\text{and}\quad
   \hhat_{A,F}(P,O)=\hhat_{E,(O)}(O)=0.
\]  
Thus $\a_f(P,O)=\d_f$ and $\hhat_{A,F}(P,O)=0$, which shows that the
implication $\qhat_{A,F}(x)>0\;\Rightarrow\;\a_f(x)=\d_f$ cannot be
reversed.
\end{remark}

\section{An example of nef heights on certain CM abelian varieties}
\label{section:abeliansurfaceexample}
In this section we illustrate Theorem~\ref{theorem:hADx0ax0} and
Corollary~\ref{corollary:Ofxdenseafxdf}(b) by working out the details
for a non-trivial example, specificaly for an abelian variety~$A$
whose endomorphism algebra $\End(A)_\QQ$ is isomorphic to a real
quadratic field~$\QQ(\s)$, where $\s^2=m$ is a positive non-square
integer. In this case the Rosati involution is the identity map, so
$\F:\NS(A)_\QQ\to\End(A)_\QQ$ is an isomorphism. We always
have~$\F_H=1_A$, and we choose a divisor~$F\in\NS(A)_\QQ$ such that
$\F_F=\s$. Then
\[
  \End(A)_\RR \cong \RR(\s) \xrightarrow[i]{\;\sim\;} \RR\times\RR,
  \qquad
  i(a+b\s) = \left(a+b\sqrt{m},a-b\sqrt{m}\right),
\]
and $D\in\NS(A)_\RR$ is nef if and only if both coordinates of
$i(\F_D)$ are non-negative; cf.\ \cite[page~210]{MR0282985}.
\par
The divisor
\[
  D = \sqrt{m}H+F \in \NS(A)_\RR
\]
satisfies 
\[
  i(\F_D) = i\left(\sqrt{m}\F_H+\F_F\right)
  = i\left(\sqrt{m}+\s\right)
  = \left(2\sqrt{m},0\right),
\]
so~$D$ is nef. Proposition~\ref{proposition:DnefiffFda2} says
that~$\F_D$ can be written in the form~$\a'\circ\a$ for
some~$\a\in\End(A)_\RR$.  In fact, we explicitly have
\[
  \F_D^2 = (\sqrt{m}+\s)^2 = 2m+2\sqrt{m}\s = 2\sqrt{m}\F_D,
\]
so 
\[
  \F_D = \a'\circ\a=\a^2 \quad\text{with $\a=(4m)^{-1/4}\F_D$.}
\]
Then $\hhat_{A,D}(x)=\hhat_{A,H}\bigl(\a(x)\bigr)$, cf.\ the
computation in the proof of Theorem~\ref{theorem:hADx0ax0},
so
\[
  \hhat_{A,D}(x)=0
  \quad\Longleftrightarrow\quad
  \a(x) = 0 \quad\text{in $A(\Qbar)_\RR$.}
\]
But
\[
  \a(x) = (4m)^{-1/4}\F_D(x) = (4m)^{-1/4}\left(\sqrt m x + \s(x)\right),
\]
so we see that
\begin{equation}
  \label{eqn:hADx0mxsx0}
  \hhat_{A,D}(x)=0
  \quad\Longleftrightarrow\quad
  \sqrt{m}x+\s(x) = 0 \quad\text{in $A(\Qbar)_\RR$.}
\end{equation}
Consider the linear transformation
\[
  T : A(\Qbar)_\RR \longrightarrow A(\Qbar)_\RR,\qquad
     x \longmapsto \left(\sqrt{m}+\s\right)(x) = \sqrt{m}x+\s(x).
\]
It follows from~\eqref{eqn:hADx0mxsx0} that $\hhat_{A,D}(x)=0$ if
and only if~$x\in\Ker(T)$. If we also assume that~$x\in A(\Qbar)_\QQ$,
then choosing a basis $\{v_i\}$ for $A(\Qbar)_\QQ$
and writing $x=\sum a_iv_i$ with $a_i\in\QQ$, we have
(all sums have finitely many nonzero terms)
\begin{align*}
  0 = T(x) 
  &= \sum_i a_iT(v_i) \\
  &= \sum_i a_i\sqrt{m}v_i + a_i\s(v_i) \\
  &= \sum_i a_i\sqrt{m}v_i + a_i\sum_j b_{ij}v_j
    \quad\text{for some $b_{ij}\in\QQ$,} \\
  &= \sum_i \left(a_i\sqrt{m}+\sum_j a_jb_{ji}\right)v_i.
\end{align*}
Since $\{v_j\}$ is a $\QQ$-basis for~$A(\Qbar)_\QQ$, it is \emph{a
fortiori} an~$\RR$-basis for~$A(\Qbar)_\RR$, so we see that
$a_i\sqrt{m}+\sum_j a_jb_{ji}=0$ for all~$i$. But $a_i,b_{ij}\in\QQ$,
while~$\sqrt{m}\notin\QQ$, so we conclude that~$a_i=0$ for
all~$i$. Hence~$x=0$ in~$A(\Qbar)_\QQ$, which is equivalent to~$x$
being a torsion point. We note that this is the appropriate conclusion
from Theorem~\ref{theorem:hADx0ax0}, since~$A$ is simple, so the
abelian subvariety~$B\subsetneq A$ must be~$B=0$, and
thus~$B^\div=A_\tors$.  Finally, we remark that an easy calculation
shows that~$D$ is an eigendivisor for every endomorphism
$f\in\End(A)$, and more precisely, that $f^*D\equiv\d_fD$. Hence
if~$x\notin A_\tors$, then we have proven that~$\hhat_{A,D}(x)>0$,
from which we conclude (as in the proof of
Corollary~\ref{corollary:Ofxdenseafxdf}(b)) that
\[
  \a_f(x)\ge\liminf \hhat_{A,D}\bigl(f^n(x)\bigr)^{1/n}
  = \liminf \bigl(\d_f^n\hhat_{A,D}(x)\bigr)^{1/n}
  = \d_f.
\]

\bibliographystyle{abbrv}

\end{document}